\documentclass{amsart}

\usepackage{amssymb}
\usepackage{enumerate}

\usepackage[foot]{amsaddr}

\newtheorem{thm}{Theorem}[section]

\newtheorem{lem}[thm]{Lemma}
\newtheorem{cor}[thm]{Corollary}

\theoremstyle{definition}
\newtheorem{defn}[thm]{Definition}

\theoremstyle{remark}
\newtheorem{example}[thm]{Example}
\newtheorem{remark}[thm]{Remark}

\newcommand{\NN}{\mathbb{N}}
\newcommand{\RR}{\mathbb{R}}  
\newcommand{\norm}[1]{\left\| {#1} \right\|}
\newcommand{\seq}[1]{\left\{ {#1} \right\}}
\newcommand{\pair}[1]{\langle {#1} \rangle}

\begin{document}

\title[Rates of convergence for weakly contractive mappings]{Rates of convergence for asymptotically weakly contractive mappings in normed spaces}

\author{Thomas Powell}
\address{Department of Computer Science, University of Bath}
\email{trjp20@bath.ac.uk}

\author{Franziskus Wiesnet}
\address{Department of Mathematics, University of Trento}
\email{Franziskus.Wiesnet@unitn.it}

\begin{abstract}
We study Krasnoselskii-Mann style iterative algorithms for approximating fixpoints of asymptotically weakly contractive mappings, with a focus on providing generalised convergence proofs along with explicit rates of convergence. More specifically, we define a new notion of being asymptotically $\psi$-weakly contractive with modulus, and present a series of abstract convergence theorems which both generalise and unify known results from the literature. Rates of convergence are formulated in terms of our modulus of contractivity, in conjunction with other moduli and functions which form quantitative analogues of additional assumptions that are required in each case. Our approach makes use of ideas from proof theory, in particular our emphasis on abstraction and on formulating our main results in a quantitative manner. As such, the paper can be seen as a contribution to the \emph{proof mining} program.
\end{abstract}

\maketitle





\section{Introduction}
\label{sec:intro}

Let $X$ be a real normed space and $T: E\to X$ a mapping defined on some closed convex set $E\subseteq X$. We say that $T$ is strongly contractive if there exists some $k\in (0,1)$ such that
\begin{equation*}
\norm{Tx-Ty}\leq (1-k)\norm{x-y}
\end{equation*}
for all $x,y\in E$. It is well known by Banach's fixed point theorem that the Picard iteration scheme $x_{n+1}=Tx_n$ converges to the unique fixpoint $q$ of $T$ with rate of convergence
\begin{equation*}
\norm{x_n-q}\leq \frac{(1-k)^n}{k}\norm{x_1-x_0}
\end{equation*}
Many important results in functional analysis are based on establishing the existence of a fixpoint together with convergence towards this fixpoint for more general classes of contractive mappings and approximation schemes. This paper will be based around one such class, the so-called weakly contractive mapping introduced by Alber and Guerre-Delabriere in \cite{alber-guerredelabriere:97:weaklycontractive}. To be more precise, a mapping $T$ is $\psi$-weakly contractive if there exists a nondecreasing function $\psi:[0,\infty)\to [0,\infty)$ which is positive on $(0,\infty)$ with $\psi(0)=0$, such that
\begin{equation*}
\norm{Tx-Ty}\leq \norm{x-y}-\psi(\norm{x-y})
\end{equation*}
for all $x,y\in E$. This generalises the class of strongly contractive mappings -- which correspond to the simple case $\psi(t)=kt$ for $k\in (0,1)$ -- and in turn represent a special uniform instance of the class of \emph{contractive} mappings which satisfy
\begin{equation*}
x\neq y\implies \norm{Tx-Ty}<\norm{x-y}
\end{equation*}
In \cite{alber-guerredelabriere:97:weaklycontractive} it is shown that under certain additional assumptions, weakly contractive mappings possess unique fixpoints $q$ which are approximated by the Picard iteration scheme. This is achieved by observing that the sequence $\lambda_n:=\norm{x_n-q}$ satisfies the recursive inequality
\begin{equation}
\label{eqn:recin:simple}
\lambda_{n+1}\leq \lambda_n-\psi(\lambda_n)
\end{equation}
from which it follows that $\lambda_n\to 0$ with the rate of convergence
\begin{equation*}
\lambda_n\leq \Phi^{-1}(\Phi(\lambda_0)-(n-1))
\end{equation*}
for $\Phi(s):=\int^s dt/\psi(t)$. Many variations on this basic result have been explored, starting in \cite{alber-guerredelabriere:97:weaklycontractive} continuing in \cite{alber-chidume-zegeye:06:nonexpansive, alber-guerredelabriere:01:projection, alber-reich-yao:03:weaklycontractive, chidume-zegeye-aneke:02:dweakly}, to name just a few. In these papers, one considers variants of the more general Krasnoselskii-Mann scheme
\begin{equation}
\label{eqn:mann:iteration}
x_{n+1}=(1-\alpha_n)x_n+\alpha_n Tx_n
\end{equation} 
where $\seq{\alpha_n}$ is some sequence of positive reals satisfying $\sum_{n=0}^\infty \alpha_n=\infty$ (setting $\alpha_n=1$ we regain the Picard iteration as a special case). A wide range of convergence theorems have been established, based on different modifications of the basic parameters, namely the mapping $T$, the iterative algorithm $\seq{x_n}$ and the underlying space $X$. For example, \cite{alber-chidume-zegeye:06:nonexpansive} introduces a class of asymptotically weakly contractive mappings which satisfy the more general property
\begin{equation*}
\norm{T^nx-T^ny}\leq \norm{x-y}-\psi(\norm{x-y})+k_n\phi(\norm{x-y})+l_n
\end{equation*} 
for $k_n,l_n\to 0$ as $n\to\infty$, and where certain bounding assumptions are required to establish convergence. Another variation is considered in \cite{chidume-zegeye-aneke:02:dweakly}, which investigates so-called $d$-weakly contractive mappings, based on the contractivity condition that for any $x,y\in E$ there exists some $j\in J(x-y)$ such that
\begin{equation*}
\pair{Tx-Ty,j}\leq \norm{x-y}^2-\psi(\norm{x-y})
\end{equation*}
where here the underlying space $X$ is uniformly smooth and $J$ is the normalized duality mapping. Finally, one can modify the iterative scheme: For example, \cite{alber-reich-yao:03:weaklycontractive} considers the perturbed iteration
\begin{equation*}
x_{n+1}=Q_n((1-\alpha_n)x_n+\alpha_nTx_n)
\end{equation*}
for sunny nonexpansive retractions $Q_n:X\to E_n$. Here, in order to establish convergence we require a number of additional properties, including uniform smoothness of $X$ and stability of the sequence $\seq{E_n}$ with respect to the Hausdorff metric. 

In each of the above cases, convergence proofs tend to involve a strategy which reduces the problem via a series of (often quite intricate) steps the following abstract recursive inequality:
\begin{equation*}
\lambda_{n+1}\leq \lambda_n-\alpha_n\psi(\lambda_n)+\gamma_n
\end{equation*}
where $\gamma_n/\alpha_n\to 0$. One can then appeal to a standard results from the theory of recursive inequalities to establish $\lambda_n\to 0$. Rates of convergence become more difficult to formulate in comparison to convergence proofs which can be reduced to the simpler inequality (\ref{eqn:recin:simple}), and are typically dependent in a subtle way on parameters which arise from the various additional assumptions that are required for convergence. 

In this paper, we use ideas from proof theory to establish a series of general convergence theorems for classes of weakly contractive mappings. These not only strengthen the aforementioned results by weakening parameters and introducing suitably abstract formulations of key properties, but bring them together as part of a unifying scheme. In addition, we provide concrete rates of convergence in all cases, appealing to the appropriate proof-theoretic moduli in order to formulate them properly.

\subsection{Applied proof theory}

Our paper can be viewed as a contribution to the logic based \emph{proof mining} program, which applies ideas and techniques from logic to mathematical proofs, with the aim of both generalising those proofs and extracting quantitative data such as rates of convergence. A comprehensive background to the area is presented in \cite{kohlenbach:08:book}, and the survey papers \cite{kohlenbach:19:nonlinear:icm,kohlenbach-oliva:03:systematic} provide a good overview of some of the many applications in analysis. Proof mining continues to be expanded to new areas of mathematics: In recent years this includes pursuit-evasion games \cite{kohlenbach-lopezacedo-nicolae:21:lionman}, differential algebra \cite{simmons:towsner:19:polyrings} and Tauberian theory \cite{powell:20:tauberian}.

No prior familiarity of proof mining is assumed in this paper, and in particular, we do not explicitly mention, nor require knowledge of, any concepts from logic. That being said, our paper bears a number of important hallmarks which betray its implicit proof theoretic flavour.

Firstly, we make use of ``proof-theoretic moduli'' in place of the traditional moduli one typically encounters in analysis, the former being based on the \emph{logical} structure of the property in question. For example, a rate of convergence in our sense is simply a function $\Psi:(0,\infty)\to \NN$ with the property that for any $\varepsilon>0$ we have $\norm{x_n-q}\leq \varepsilon$ for all $n\geq \Psi(\varepsilon)$, as opposed to a function $F:\NN\to (0,\infty)$ satisfying $\norm{x_n-q}\leq F(n)$ for all $n\in\NN$ and $F(n)\to 0$ as $n\to\infty$. We work with similar moduli for representing properties such as uniform continuity or uniform smoothness, together with new moduli defined for this first time here which represent sequences of mappings which are asymptotically weakly contractive. By the using them to represent assumptions in our theorems, we are able to extract rates of convergence over the structure of proofs in terms of these moduli. In fact precisely because the moduli mirror the logical structure of the property in question, they propagate through the logical structure of the convergence proof resulting in the aforementioned rate. In many cases, our rates of convergence can be reformulated and then directly compared to those which have been given in the literature, but we are also able to give explicit rates of convergence for general theorems which, to the best of our knowledge, are new.

Another feature of our approach in that our main results were obtained through the careful analysis of existing proofs, where we sought to dispose of superfluous assumptions (which, for instance, were only required to establish \emph{existence} of a fixed point) and provide abstract versions of others, resulting in generalised theorems in which certain premises have been weakened and others eliminated altogether. As a direct result of this, it is sometimes the case that distinct theorems which appear separately in the literature can each be presented as direct corollaries of our results. In this way, our work forms a unifying scheme in which different convergence theorems can be classified and compared within a single framework. 

Finally, as a result of our effort to give suitable abstract presentations of key properties, we present several new classes of mappings which generalise various notions of being $\psi$-weakly contractive. For instance, in Section \ref{sec:cs1} we formulate our main result in terms of what we call quasi asymptotically weakly contractive mappings, a class which contains the totally asymptotically weakly contractive mappings of Alber et al. \cite{alber-chidume-zegeye:06:nonexpansive} and also the sequences of approximate weakly contractive mappings studied in \cite{alber-guerredelabriere:01:projection}. Similarly, in Section \ref{sec:cs2} we present as asymptotic counterpart of the class of $d$-weakly contractive mappings.

In this paper we focus on a relatively small number of representative case studies which exemplify our proof theoretic approach to convergence theorems for mappings of weakly contractive type. It is certainly the case that there are plenty of other classes of mappings and associated convergence theorems which could be abstracted and generalised in a similar fashion, and we leave an exploration of more recent results in this area to future work (cf. Section \ref{sec:conc}).

\section{Basic moduli}
\label{sec:basic}

We begin by introducing some of the proof-theoretic moduli which form the core of this paper, and in doing so briefly outline our abstract way of looking at classes of weakly contractive mappings.

\subsection{Rates of convergence and divergence}
\label{sec:basic:convdiv}

The central concept of this paper is the rate of convergence.

\begin{defn}
\label{def:rate:conv}
Suppose that $\seq{x_n}$ is a sequence in some metric space $(X,d)$ with $\lim_{n\to\infty} x_n= q$. A rate of convergence for $\seq{x_n}$ is any function $f:(0,\infty)\to\NN$ with the property that for any $\varepsilon>0$ we have $d(x_n,q)\leq \varepsilon$ for all $n\geq f(\varepsilon)$. 
\end{defn}
\begin{example}
If $r\in (0,1)$ is such that $d(x_{n+1},q)\leq r d(x_n,q)$ for all $n\in\NN$ then $\seq{x_n}$ converges with rate
\begin{equation*}
f(\varepsilon):=\frac{\log(\varepsilon)-\log(d(x_0,q))}{\log(r)}
\end{equation*}
\end{example}
Because we consider algorithms based on iterative schemes of the form (\ref{eqn:mann:iteration}) which rely on the condition that the sum of the coefficients $\sum_{n=0}^\infty\alpha_n$ diverges, we need a corresponding modulus for this property.
\begin{defn}
\label{def:rate:div}
Suppose that $\seq{\alpha_n}$ is a sequence of nonnegative reals such that $\sum_{n=0}^\infty\alpha_n=\infty$. A rate of divergence for this series is a function $r:\NN\times(0,\infty)\to (0,\infty)$ such that for all $N\in\NN$ and $x\in (0,\infty)$ we have $r(N,x)\geq N$ and
\begin{equation*}
\sum_{n=N}^{r(N,x)}\alpha_n > x
\end{equation*}
\end{defn} 
\begin{example}
If $\alpha_n=1$ for all $n\in\NN$ -- in which case the Mann scheme (\ref{eqn:mann:iteration}) would reduce to the usual Picard iteration -- a rate of divergence is given by
\begin{equation*}
r(N,x):=\lceil x+N\rceil
\end{equation*}
\end{example}
\begin{remark}
An alternative way of representing divergence of a series would be via a function $g:(0,\infty)\to \NN$ satisfying
\begin{equation*}
\sum_{n=0}^{g(x)}\alpha_n> x
\end{equation*}
in which case, a rate of divergence in our sense could be defined as
\begin{equation*}
r(N,x):=g\left(x+\sum_{n=0}^{N-1}\alpha_n\right)
\end{equation*}
However, in what follows we will work directly with the modulus $r$.
\end{remark}

\subsection{Moduli for smooth and uniformly smooth spaces}
\label{sec:basic:smooth}

A Banach space $X$ is said to be uniformly smooth if for any $\varepsilon>0$ there exists some $\delta>0$ such that for all $x,y\in X$ with $\norm{x}=1$ and $\norm{y}\leq \delta$ we have
\begin{equation}
\label{eqn:unifsmooth}
\norm{x+y}+\norm{x-y}\leq 2+\varepsilon\norm{y}
\end{equation}
A function $\tau:(0,\infty)\to (0,\infty)$ which for any input $\varepsilon$ returns a $\delta$ satisfying (\ref{eqn:unifsmooth}) is called a modulus of uniform smoothness. This is not to be confused with the so-called modulus of smoothness, which is defined by
\begin{equation*}
\rho_X(\delta):=\sup\left\{\frac{\norm{x+y}+\norm{x-y}}{2}-1\; : \; \norm{x}=1, \norm{y}=\delta\right\}
\end{equation*}
and satisfies 
\begin{equation*}
\lim_{\delta\to 0}\frac{\rho_X(\delta)}{\delta}=0
\end{equation*}
iff $X$ is uniformly smooth.

\subsection{Notions of contractivity}
\label{sec:basic:contr}

Let $X$ be a normed space and $T:E\to X$ a mapping. We begin our discussion of contractivity by presenting and comparing three closely related notions:
\begin{defn}
\label{defn:contr}
\begin{enumerate}[(a)]

\item\label{defn:contr:basic} $T$ is contractive if for all $x,y\in E$:
\begin{equation*}
x\neq y\implies \norm{Tx-Ty}<\norm{x-y}
\end{equation*}

\item\label{defn:contr:modulus} $T$ is contractive with modulus $\tau:(0,\infty)\to (0,\infty)$ if for all $x,y\in E$ and $\varepsilon>0$ we have
\begin{equation*}
\norm{x-y}\geq \varepsilon\implies \norm{Tx-Ty}+\tau(\varepsilon)\leq \norm{x-y}
\end{equation*}

\item\label{defn:contr:weakly} $T$ is weakly contractive if there exists some nondecreasing map $\psi:[0,\infty)\to [0,\infty)$ which is positive on $(0,\infty)$ and has $\psi(0)=0$ such that for all $x,y\in E$ we have
\begin{equation*}
\norm{Tx-Ty}\leq \norm{x-y}-\psi(\norm{x-y})
\end{equation*}

\end{enumerate}
\end{defn}
Contractive mappings in the sense of (\ref{defn:contr:basic}) have been widely studied, also in the more general setting of metric spaces, where Edelstein's fixed point theorem \cite{edelstein:62:contractive} states that whenever $T:E\to E$ for some compact $E$, then $T$ has a unique fixed point. For the other notions of contractivity, it is clear that if $T$ is weakly contractive w.r.t. $\psi$ then it is also contractive with modulus $\psi$, and if $T$ is contractive with any modulus, it is also contractive, and thus we easily arrive at
\begin{equation*}
\mbox{(\ref{defn:contr:weakly})}\implies\mbox{(\ref{defn:contr:modulus})}\implies \mbox{(\ref{defn:contr:basic})}
\end{equation*}
The other directions are more interesting. Contractivity in the sense of (\ref{defn:contr:basic}) can be formulated more explicitly by the logical formula
\begin{equation}
\label{eqn:contractive}
\forall x,y\in E\; \forall \varepsilon>0\; \exists \delta>0\; (\norm{x-y}\geq \varepsilon\implies \norm{Tx-Ty}+\delta\leq \norm{x-y})
\end{equation}
and a contractive mappings possesses a modulus precisely when for each $\varepsilon>0$ there exists a $\delta>0$ satisfying (\ref{eqn:contractive}) uniformly in the parameters $x,y\in E$ (a standard argument using sequential compactness shows that this is always the case when, for example, $E$ is compact). Thus a contractive mapping with a modulus is one that is contractive in a uniform way, and such mappings form a particularly elegant class to study, especially from a quantitative perspective. Indeed, in \cite{kohlenbach-oliva:03:systematic} it is shown that in the case that $E$ is compact, a modulus of contractivity can be characterised proof theoretically as the so-called monotone functional interpretation of the statement that $T$ is contractive, and a such a modulus is used to formulate a rate of convergence for Edelstein's fixed point theorem. A particular kind of contractive mapping with modulus -- so-called almost uniform contractions -- is also considered in the context of Bishop-style constructive analysis in \cite{bridges-etal:92:contractive}. 

In the case that a modulus of contractivity $\tau$ for $T$ is nondecreasing and strictly positive, then $T$ is also weakly contractive w.r.t. $\tau$. Otherwise, defining $\psi(0)=0$ and
\begin{equation*}
\psi(\varepsilon):=\inf\{\tau(\mu)\; : \; \varepsilon\leq \mu\},
\end{equation*} 
whenever this infimum is always strictly positive, we have that $\psi$ is nondecreasing and $T$ is $\psi$-weakly contractive. Thus notions (\ref{defn:contr:modulus}) and (\ref{defn:contr:weakly}) are closely connected. Interestingly however, while contractive mappings with moduli have been studied in proof theoretic approaches to analysis, weakly contractive mappings were introduced quite independently in \cite{alber-guerredelabriere:97:weaklycontractive}. Here, some addition conditions on $\psi$ are assumed, which allow us to prove that whenever $X$ is a Hilbert space (or more generally a uniformly smooth Banach space) then $T$ has a fixed point, even in cases where $E$ is not compact.

\subsection{Asymptotically weakly contractive mappings}
\label{sec:basic:asymp}

Rather than simple contractive mappings, we will be interested in the classes of mappings which are weakly contractive in an asymptotic sense. There are many different ways to define such classes. A simple example would mappings $T:E\to X$ which satisfy the property
\begin{equation}
\label{eqn:asymp:contr:simple1}
\norm{T^nx-T^ny}\leq (1+k_n)\norm{x-y}-\psi(\norm{x-y})
\end{equation}
where $\seq{k_n}$ is some sequence of reals with $k_n\to 0$. Another example would be given by
\begin{equation}
\label{eqn:asymp:contr:simple2}
\norm{T^nx-T^ny}\leq \norm{x-y}-\psi(\norm{x-y})+l_n
\end{equation}
where $l_n\to 0$. More elaborate and general classes which include both of these types are considered in e.g. \cite{alber-chidume-zegeye:06:nonexpansive} and \cite{alber-guerredelabriere:97:weaklycontractive}. What these variants have in common is that they concern sequences of mappings $\seq{A_n}$ which become weakly contractive in the limit (where often this sequence is simply taken to be the iterates of a single mapping i.e. $A_n:=T^n$ for some $T$). In a similar vein to what we have discussed already, we can express this notion as a simple $\varepsilon$/$\delta$ property and provide a corresponding modulus, which measures how quickly such sequences become weakly contractive. 
\begin{defn}
\label{defn:asymp:contr}
Let $A_n:E\to X$ be a sequence of mappings and $\psi:[0,\infty)\to [0,\infty)$ some nondecreasing function with $\psi(0)=0$. We say that the sequence $\seq{A_n}$ is asymptotically $\psi$-weakly contractive if for all $\delta, b>0$ there exists some $m$ such that
\begin{equation*}
\norm{x-y}\leq b\implies \norm{A_nx-A_ny}\leq \norm{x-y}-\psi(\norm{x-y})+\delta
\end{equation*}
for all $n\geq m$. A function $\sigma:(0,\infty)\times (0,\infty)\to \NN$ which produces such an $m$ in arguments $\delta, b$ will be called a modulus of asymptotic $\psi$-weak contractivity. 
\end{defn}
\begin{example}
If $T$ is a mapping which satisfies the simple property (\ref{eqn:asymp:contr:simple1}) then $\seq{T^n}$ is asymptotically $\psi$-weakly contractive in our sense with modulus
\begin{equation*}
\sigma(\delta,b):=f\left(\frac{\delta}{b}\right)
\end{equation*}
where $f$ is a rate of convergence for $k_n\to 0$. Similarly, in the case of property (\ref{eqn:asymp:contr:simple2}) a modulus is given by $\sigma(\delta):=g(\delta)$ for $g$ a rate of convergence for $l_n\to 0$. In the second case, the modulus is independent of $b$: Such moduli play a special role here and will examined more closely later. Note that for ordinary $\psi$-weakly contractive mapping in the sense of Definition \ref{defn:contr} (\ref{defn:contr:weakly}) we can simply set $\sigma(\delta,b)=0$. 
\end{example}
Later on in this paper, we will show that not only do a number of previous formulations of asymptotic weak contractivity form instances of this general definition (cf. Section \ref{sec:cs1}), but existing results from the literature which concern only weakly contractive mappings can be generalised to the asymptotic case (cf. Sections \ref{sec:cs2} and \ref{sec:cs3}). Moreover, Definition \ref{defn:asymp:contr} itself will be further varied and generalised. For example, many (but not all) of the results we prove later work when the relevant notion of asymptotic $\psi$-weak contractivity is replaced by what we will call \emph{quasi} asymptotic $\psi$-weak contractivity relative to some fixed point $q$ i.e.
\begin{equation*}
\norm{x-q}\leq b\implies \norm{A_nx-q}\leq \norm{x-q}-\psi(\norm{x-q})+\delta 
\end{equation*}
\begin{remark}
A related notion of being asymptotically contractive in the setting of arbitrary complete metric spaces $(X,d)$ is given by Kirk \cite{kirk:03:asymptotic-contractions}, where a corresponding convergence result for Picard iterates is proven. This has been analysed from a proof theoretic standpoint first by Gerhardy \cite{gerhardy:06:kirk} and then Briseid \cite{briseid:07:asymptotic}, both of whom develop quantitative notions of being asymptotically contractive similar in spirit to our Definition \ref{defn:asymp:contr} (cf. Definition 2 of \cite{gerhardy:06:kirk} and Definition 2.1 of \cite{briseid:07:asymptotic}). However, our focus here is different from theirs, as we are interested in convergence results for Mann iteration schemes in normed spaces, and the proofs that we analyse have a very different character.
\end{remark}

\section{Quantitative recursive inequalities}
\label{sec:recursive}

All of the proofs in the main part of this paper utilise an abstract theory of recursive inequalities, a quantitative analysis of which is not only crucial in obtaining our rates of convergence, but forms a unifying scheme which, together with the abstract notions of contractivity discussed above, allowing us to bring together several distinct convergence results from the literature. In this section we present the core quantitative convergence results which will be needed later, and also take the opportunity to compare our rates of convergence (formulated in terms of moduli) with those which occur in literature (typically formulated in terms of bounding functions). 

\subsection{Recursive inequalities and asymptotic contractivity}

For illustrative purposes and to motivate the results that follow, let us consider a mapping $T:E\to E$ which satisfies the simplified variant (\ref{eqn:asymp:contr:simple2}) of being total asymptotically weakly contractive, and suppose that $q\in E$ is some fixpoint of $T$. Let us suppose that $\seq{x_n}$ is the standard Mann iterative scheme for approximating this fixpoint i.e. assuming now that $E$ is convex,
\begin{equation*}
x_{n+1} = (1-\alpha_n)x_n+\alpha_nT^nx_n
\end{equation*}
for some sequence $\seq{\alpha_n}$ of reals in $(0,1]$ satisfying $\sum_{n=0}^\infty \alpha_n=\infty$. In order to show that this algorithm convergences strongly to the fixpoint $q$, we observe that
\begin{equation*}
\begin{aligned}
\norm{x_{n+1}-q}&\leq (1-\alpha_n)\norm{x_n-q}+\alpha_n\norm{T^nx_n-T^nq}\\
&\leq (1-\alpha_n)\norm{x_n-q}+\alpha_n(\norm{x_n-q}-\psi(\norm{x_n-q})+l_n)\\
&\leq \norm{x_n-q}-\alpha_n\psi(\norm{x_n-q})+\alpha_nl_n
\end{aligned}
\end{equation*}
and therefore the sequence $\mu_n:=\norm{x_n-q}$ satisfies the following recursive inequality:
\begin{equation}
\label{eqn:recineq:basic}
\mu_{n+1}\leq \mu_n-\alpha_n\psi(\mu_n)+\alpha_nl_n
\end{equation}
It is well known that any positive sequence $\seq{\mu_n}$ satisfying (\ref{eqn:recineq:basic}) must convergence to zero. This follows from a general theory of recursive inequalities of this kind, which have been studied in e.g. \cite{alber:83:inequalities, alber-reich:94:panamerican}. In particular, a detailed account of such convergence results along with proofs is given in \cite{alber-iusem:01:subgradient}, where convergence of $\seq{\mu_n}$ above follows as a special case of their Lemma 2.5 (cf. Remark 2.6 of \cite{alber-iusem:01:subgradient}).

Many convergence theorems involving variants of weakly contractive mappings, involve a similar (though typically much more complex) reduction to this or a similar recursive scheme, and thus our starting point is to provide a computational analysis of the relevant abstract convergence theorems which will be used in later sections.


\subsection{Main quantitative lemmas}

We now formulate some key quantitative convergence results based on an abstract formulation of the recursive inequality (\ref{eqn:recineq:basic}). These bear similarities to the quantitative analysis of the closely related inequality
\begin{equation*}
\label{eqn:recineq:acc}
\mu_{n+1}\leq \mu_n-\alpha_n\psi(\mu_{n+1})+\alpha_nl_n
\end{equation*}
which can be found as \cite[Lemma 1]{kohlenbach-koernlein:11:pseudocontractive} and \cite[Lemma 3.4]{kohlenbach-powell:20:accretive}, but the respective proofs that $\mu_n\to 0$ are somewhat different, and so the results which follow are new. Note that in addition to a rate of convergence, Lemma \ref{res:recineq:first} establishes a new proof that $\mu_n\to 0$, which is somewhat different to that found in e.g. \cite{alber-guerredelabriere:97:weaklycontractive} or \cite{alber-iusem:01:subgradient}.
\begin{lem}
\label{res:recineq:first}
Let $\seq{\mu_n}$ be a sequence of nonnegative real numbers, $\seq{\alpha_n}$ a bounded sequence of positive numbers with $\sum_{n=0}^\infty \alpha_n=\infty$ and $\psi:[0,\infty)\to [0,\infty)$ be a nondecreasing function which is positive on $(0,\infty)$. Suppose that that for all $\delta>0$ there exists some $m\in\NN$ such that 
\begin{equation}
\label{eqn:recineq:abs}
\mu_{n+1}\leq \mu_n-\alpha_n(\psi(\mu_n)-\delta)
\end{equation}
for all $n\geq m$. Then $\mu_n\to 0$. Moreover, if $\alpha>0$, $r:(0,\infty)\times (0,\infty)\to\NN$ and $N:(0,\infty)\to \NN$ are such that
\begin{itemize}

\item $\alpha_n\leq\alpha$ for all $n\in\NN$,

\item $r$ is a rate of divergence for $\sum_{n=0}^\infty\alpha_n$,

\item for any $\delta>0$ the inequality (\ref{eqn:recineq:abs}) holds for $n\geq N(\delta)$,

\end{itemize}
then $\mu_n\to 0$ with the following rate of convergence:
\begin{equation*}
\Phi_{\psi,c,\alpha,r,N}(\varepsilon):=r\left(N\left(\frac{1}{2}\min\left\{\psi\left(\frac{\varepsilon}{2}\right),\frac{\varepsilon}{\alpha}\right\}\right),2\int_{\varepsilon/2}^c \frac{dt}{\psi(t)} \right)+1
\end{equation*}
where $c$ is any upper bound on $\seq{\mu_n}$ (but cf. Remark \ref{rem:uniform:bound}).
\end{lem}

\begin{proof}
Fix $\varepsilon>0$ and let 
\begin{equation*}
N_0:=N\left(\min\left\{\frac{\psi(\varepsilon)}{2},\frac{\varepsilon}{\alpha}\right\}\right).
\end{equation*}
Then for all $n\geq N_0$ we have both
\begin{equation}
\label{eqn:recineq:use1}
\mu_{n+1}\leq \mu_n-\alpha_n\left(\psi(\mu_n)-\frac{\psi(\varepsilon)}{2}\right)
\end{equation}
and (using that $\alpha_n\leq \alpha$)
\begin{equation}
\label{eqn:recineq:use2}
\mu_{n+1}\leq \mu_n-\alpha_n\psi(\mu_n)+\varepsilon.
\end{equation}
Now let $l\geq N_0$ be arbitrary and suppose that $\varepsilon\leq \mu_n$ for all $N_0\leq n\leq l+1$. Then by monotonicity of $\psi$ we have $\psi(\varepsilon)\leq \psi(\mu_n)$ and thus by (\ref{eqn:recineq:use1}) it follows that
\begin{equation*}
\mu_{n+1}\leq \mu_n-\alpha_n\left(\psi(\mu_n)-\frac{\psi(\mu_n)}{2}\right)=\mu_n-\alpha_n\frac{\psi(\mu_n)}{2}
\end{equation*}
and therefore
\begin{equation*}
\frac{1}{2}\sum_{n=N_0}^l\alpha_n\leq \sum_{n=N_0}^l \left(\frac{\mu_n-\mu_{n+1}}{\psi(\mu_n)}\right)\leq \sum_{n=N_0}^l\left(\int_{\mu_{n+1}}^{\mu_n} \frac{dt}{\psi(t)}\right)=\int_{\mu_{l+1}}^{\mu_{N_0}}\frac{dt}{\psi(t)}
\end{equation*}
where for the second inequality we observe that for $N_0\leq n\leq l$ we have $0<\mu_{n+1}<\mu_n$ and thus the function $1/\psi(t)$ is well-defined, positive valued and monotonically decreasing on $[\mu_{n+1},\mu_n]$, and hence integrable with
\begin{equation*}
\frac{\mu_n-\mu_{n+1}}{\psi(\mu_n)}\leq\int_{\mu_{n+1}}^{\mu_n}\frac{dt}{\psi(t)}
\end{equation*}
Finally, since $\varepsilon\leq \mu_{l+1}<\mu_{N_0}\leq c$ (recall that  $c$ is an upper bound on $\seq{\mu_n}$) we have
\begin{equation*}
\sum_{n=N_0}^l\alpha_n\leq 2\int_{\varepsilon}^c\frac{dt}{\psi(t)}
\end{equation*}
which is false for $l:=r(N_0,2\int_{\varepsilon}^c\frac{dt}{\psi(t)})$. Therefore our assumption that $\varepsilon\leq \mu_n$ for all $N_0\leq n\leq l+1$ leads to a contradiction for this value of $l$, or in other words, there exists some $n\leq l+1$ such that $\mu_n<\varepsilon$. We now claim that in fact $\mu_k\leq 2\varepsilon$ for all $k\geq n$. This is shown by induction, where the base case is obvious and for the induction step we deal with two cases. Firstly, if $\varepsilon\leq \mu_k\leq 2\varepsilon$ then since $k\geq N_0$ and $\psi(\varepsilon)\leq \psi(\mu_k)$ it follows from (\ref{eqn:recineq:use1}) that
\begin{equation*}
\mu_{k+1}\leq \mu_k-\alpha_k\left(\psi(\mu_k)-\frac{\psi(\varepsilon)}{2}\right)=\mu_k-\alpha_k\frac{\psi(\varepsilon)}{2}<\mu_k\leq 2\varepsilon
\end{equation*}
On the other hand, if $\mu_k<\varepsilon$ then from (\ref{eqn:recineq:use2}) we have
\begin{equation*}
\mu_{k+1}\leq \mu_k-\alpha_k\psi(\mu_k)+\varepsilon\leq \mu_k+\varepsilon<2\varepsilon
\end{equation*}
This proves the claim, and thus in particular it follows that $\mu_k\leq 2\varepsilon$ for all $k\geq l+1$ for $l$ defined as above. Writing out $l$ in full and substituting $\varepsilon\mapsto \frac{\varepsilon}{2}$ gives us the rate of convergence, as $\varepsilon$ was arbitrary throughout.
\end{proof}
\begin{remark}
\label{rem:uniform:bound}
The uniform bound $\mu_n\leq c$ used to formulate Lemma \ref{res:recineq:first} is not strictly necessary, but it does allow us to provide a simplified rate of convergence which is independent of $\seq{\mu_n}$. Without such a bound, our rate of convergence (now depending on $\seq{\mu_n}$) would be given by
\begin{equation*}
\Phi_{\psi,\seq{\mu_n},\alpha,r,N}(\varepsilon):=r\left(M_{\psi,\alpha}(\varepsilon),2\int_{\varepsilon/2}^{\mu_{M_{\psi,\alpha}(\varepsilon)}} \frac{dt}{\psi(t)} \right)+1
\end{equation*}
for $M_{\psi,\alpha}(\varepsilon):=N\left(\frac{1}{2}\min\left\{\psi\left(\frac{\varepsilon}{2}\right),\frac{\varepsilon}{\alpha}\right\}\right)$.
\end{remark}

\begin{remark}
A more common formulation of the convergence result contained in Lemma \ref{res:recineq:first} is to suppose that
\begin{equation}
\label{eqn:recineq:standard}
\mu_{n+1}\leq \mu_n-\alpha_n\psi(\mu_n)+\gamma_n
\end{equation}
for all $n\in\NN$, where now $\seq{\gamma_n}$ is some sequence of reals with $\gamma_n/\alpha_n\to 0$. But if $N$ is a rate of convergence for $\gamma_n/\alpha_n\to 0$ then for any $\delta>0$ we have 
\begin{equation*}
\mu_{n+1}\leq \mu_n-\alpha_n\psi(\mu_n)+\alpha_n\delta
\end{equation*}
for all $n\geq N(\delta)$, which is precisely the scheme (\ref{eqn:recineq:abs}). Therefore in this case we also have that $\mu_n\to 0$ with the rate of convergence given in Lemma \ref{res:recineq:first}.
\end{remark}

Before we move on, we state a useful generalisation of Lemma \ref{res:recineq:first} that we will use later, and which is based on the following well-known variant of (\ref{eqn:recineq:standard}):
\begin{equation*}
\mu_{n+1}\leq (1+\beta_n)\mu_n-\alpha_n\psi(\mu_n)+\gamma_n
\end{equation*}
where now $\seq{\beta_n}$ is a sequence of nonnegative reals with $\prod_{n=0}^\infty (1+\beta_n)<\infty$.
\begin{lem}
\label{res:recineq:second}
Let $\seq{\mu_n}$ be a sequence of nonnegative real numbers, $\seq{\alpha_n}$ a bounded sequence of positive numbers with $\sum_{n=0}^\infty \alpha_n=\infty$, $\seq{\beta_n}$ a sequence of nonnegative numbers with $\prod_{n=0}^\infty (1+\beta_n)<\infty$, and $\psi:[0,\infty)\to [0,\infty)$ be a nondecreasing function which is positive on $(0,\infty)$. Suppose that for all $\delta>0$ there exists some $m\in\NN$ such that 
\begin{equation}
\label{eqn:recineq:beta}
\mu_{n+1}\leq (1+\beta_n)\mu_n-\alpha_n(\psi(\mu_n)-\delta)
\end{equation}
for all $n\geq m$. Then $\mu_n\to 0$. Moreover, if $\alpha>0$, $d\geq 1$, $r:(0,\infty)\times (0,\infty)\to \NN$ and $N:(0,\infty)\to \NN$ are such that
\begin{itemize}

\item $\alpha_n\leq \alpha$ for all $n\in\NN$,

\item $\prod_{i=0}^n (1+\beta_i)\leq d$ for all $n\in\NN$,

\item $r$ is a rate of divergence for $\sum_{n=0}^\infty\alpha_n$,

\item for any $\delta>0$ the inequality (\ref{eqn:recineq:beta}) holds for $n\geq N(\delta)$,

\end{itemize}
then $\mu_n\to 0$ with the following rate of convergence:
\begin{equation*}
\Phi_{\psi,c,\alpha,d,r,N}(\varepsilon):=r\left(N\left(\frac{1}{2d}\min\left\{\psi\left(\frac{\varepsilon}{2d}\right),\frac{\varepsilon}{\alpha}\right\}\right),2d\int_{\varepsilon/2d}^c \frac{dt}{\psi(t)} \right)+1
\end{equation*}
where $c$ is any upper bound on $\seq{\mu_n}$.
\end{lem}

\begin{proof}
We use a standard trick to reduce (\ref{eqn:recineq:beta}) to (\ref{eqn:recineq:abs}), see e.g. Lemma 3.3. of \cite{alber-chidume-zegeye:06:nonexpansive}. Define $\lambda_0:=\mu_0$ and $\lambda_n:=\mu_n/\prod_{i=0}^{n-1}(1+\beta_i)$ for $n>0$. Then for any $\delta>0$, for all $n\geq N(\delta)$ we have
\begin{equation*}
\lambda_{n+1}\leq \lambda_n-\frac{\alpha_n \psi(\mu_n)}{\prod_{i=0}^{n}(1+\beta_i)}+\frac{\alpha_n\delta}{\prod_{i=0}^n(1+\beta_i)}\leq \lambda_n-\alpha_nd^{-1}\psi(\mu_n)+\alpha_n\delta
\end{equation*}
using that $1\leq \prod_{i=0}^n (1+\beta_i)\leq d$. Moreover, since $\mu_n=\lambda_n\prod_{i=0}^{n-1}(1+\beta_i)\geq \lambda_n$, by monotonicity of $\psi$ we have $\psi(\mu_n)\geq\psi(\lambda_n)$ and thus
\begin{equation}
\label{eqn:rec:subs}
\lambda_{n+1}\leq \lambda_n-\alpha_n(\phi(\lambda_n)-\delta)
\end{equation}
for $\phi(t)=\psi(t)/d$, which is clearly also nondecreasing and positive on $(0,\infty)$. Observing finally that $\lambda_n\leq \mu_n\leq c$ for all $n\in\NN$, we can apply Lemma \ref{res:recineq:first} to $\lambda_n$ with parameters $\phi$, $c$, $\alpha$, $r$ and $N$ to establish a rate of convergence for $\lambda_n\to 0$. Finally, noting that $\mu_n\leq d\lambda_n$, we modify this rate of convergence with the substitution $\varepsilon\mapsto \frac{\varepsilon}{d}$ to obtain the stated rate of convergence for $\mu_n\to 0$.
\end{proof}

\subsection{Reformulation in terms of traditional rates of convergence}

We now give a rough translation of our main quantitative results so that they are phrased in terms of direct rates of convergence in the sense traditionally encountered in the literature, where we seek some explicit bounding function $f:\NN\to (0,\infty)$ with $\lim_{n\to\infty} f(n)=0$ such that $\mu_n\leq f(n)$ for all $n\in\NN$. Being able to provide a closed expression of this kind typically requires additional assumptions, such as the existence of inverse functions, and so our ``proof theoretic" formulations above are preferred. Nevertheless, the translation we provide in this section facilitates a direct comparison with known convergence rates in the literature. 

We emphasise that aside from allowing us to compare our proof theoretic rates of convergence with existing rates in cases where those have been given, Lemma \ref{res:translate} below is not required anywhere else, and in particular is not used to establish any of our main results.
\begin{defn}
Let $N:(0,\infty)\to \NN$ be a nonincreasing function. A continuous bounding function $\tilde{N}:(0,\infty)\to [0,\infty)$ for $N$ is defined to be any function which is continuous, nonincreasing and satisfies $N(\varepsilon)\leq \tilde{N}(\varepsilon)$ for all $\varepsilon\in (0,\infty)$. Since $N$ can be viewed as a simple step function in any interval, it is clear that such an $\tilde{N}$ always exists.
\end{defn}
\begin{lem}
\label{res:translate}
Suppose that $\seq{\mu_n}$, $\seq{\alpha_n}$, $\seq{\beta_n}$, $\psi$, $\alpha$, $d$, $N$ and $c$ all satisfy the assumptions of Lemma \ref{res:recineq:second}. Define
\begin{equation*}
\Psi(x):=\int^x\frac{dt}{\psi(t)}
\end{equation*}
and assuming w.l.o.g. that $N$ is nonincreasing, define $\tilde{N}:(0,\infty)\to [0,\infty)$ to be a continuous bounding function for $N$. Now define $F:(0,\infty)\to \RR$ by
\begin{equation*}
F(\varepsilon):= 2d\cdot \Psi\left(\frac{\varepsilon}{2d}\right)-\alpha\cdot \tilde{N}\left(\frac{1}{2d}\min\left\{\psi\left(\frac{\varepsilon}{2d}\right),\frac{\varepsilon}{\alpha}\right\}\right)
\end{equation*}
which must be strictly increasing and continuous, and hence invertible on its range. Then $F(\varepsilon)\to -\infty$ as $\varepsilon\to 0$ and for $n\in\NN$ sufficiently large we have
\begin{equation*}
\mu_n\leq F^{-1}\left(2d\Psi(c)-\sum^{n-2}_{i=0}\alpha_i\right)
\end{equation*}
\end{lem}

\begin{proof}
We first show that $F(\varepsilon)\to -\infty$ as $\varepsilon\to 0$. This is obviously true if $\tilde N(\varepsilon)\to +\infty$ as $\varepsilon\to 0$, so we now assume that this is not the case i.e. there is some $k\in\NN$ such that $\tilde N(\varepsilon)\leq k$ and thus $N(\varepsilon)\leq k$ for all $\varepsilon\in (0,\infty)$. Defining $\lambda_n$ as in the proof of Lemma \ref{res:recineq:second}, then by definition of $N$ and (\ref{eqn:rec:subs}) it follows that 
\begin{equation}
\label{eqn:translate}
\lambda_{n+1}\leq \lambda_n-\alpha_nd^{-1}\psi(\lambda_n)+\alpha_n\delta
\end{equation}
for $n\geq N(\delta)$ and any $\delta>0$. But since $k\geq N(\delta)$ it follows that (\ref{eqn:translate}) holds for all $n\geq k$ independent of $\delta>0$, and therefore we actually have
\begin{equation}
\label{eqn:translate0}
\lambda_{n+1}\leq \lambda_n-\alpha_nd^{-1}\psi(\lambda_n)
\end{equation} 
for all $n\geq k$. Analogously to the proof of Lemma \ref{res:recineq:first}, rearranging (\ref{eqn:translate0}) and summing up to some arbitrary $m>k$ we obtain
\begin{equation*}
d^{-1}\sum_{n=k}^{m-1}\alpha_n\leq \sum_{n=k}^{m-1}\frac{\lambda_n-\lambda_{n+1}}{\psi(\lambda_n)}\leq \sum_{n=k}^{m-1}\int_{\lambda_{n+1}}^{\lambda_n}\frac{dt}{\psi(t)}=\int^{\lambda_k}_{\lambda_{m}}\frac{dt}{\psi(t)}=\Psi(\lambda_k)-\Psi(\lambda_m)
\end{equation*}
Letting $m\to \infty$ it follows from $\sum_{n=k}^{\infty} \alpha_n=\infty$ that $\Psi(\lambda_m)\to -\infty$. Since $\lambda_m\to 0$ and $\Psi$ is monotonic we can infer that $\Psi(\varepsilon)\to -\infty$ and thus $F(\varepsilon)\to -\infty$ as $\varepsilon\to 0$.

Now, assume that $n\in\NN$ is sufficiently large so that 
\begin{equation}
\label{eqn:sufficiently:large}
2d\Psi(c)-\sum^{n-2}_{i=0}\alpha_i\in F(0,\infty)
\end{equation}
We can make ``sufficiently large'' precise here by noting that $(-\infty,F(1)]\subset F(0,\infty)$ and so (\ref{eqn:sufficiently:large}) holds for all $n\geq m$ where $m$ is such that $\sum^{m-2}_{i=0}\alpha_i\geq 2d\Psi(c)-F(1)$. Now define 
\begin{equation*}
\varepsilon_n:=F^{-1}\left(2d\Psi(c)-\sum^{n-2}_{i=0}\alpha_i\right)
\end{equation*}
so that our aim becomes establishing $\mu_n\leq \varepsilon_n$ for all $n\in\NN$. By Lemma \ref{res:recineq:second} we have that $\mu_m\leq\varepsilon_n$ for all $m\geq \Phi_{\psi,c,\alpha,d,r,N}(\varepsilon_n)$ for $\Phi_{\psi,c,\alpha,d,r,N}$ as defined in that lemma, where $r$ is any rate of divergence for $\sum_{i=0}^\infty\alpha_i$. So it suffices to show that $n\geq \Phi_{\psi,c,\alpha,d,r,N}(\varepsilon_n)$, or in other words,
\begin{equation}
\label{eqn:goal0}
n\geq r\left(N\left(\frac{1}{2d}\min\left\{\psi\left(\frac{\varepsilon_n}{2d}\right),\frac{\varepsilon_n}{\alpha}\right\}\right),2d\int_{\varepsilon_n/2d}^c \frac{dt}{\psi(t)} \right)+1
\end{equation}
for some rate of divergence $r$. Suppose therefore that $r$ is given by
\begin{equation*}
r(N,x):=\mbox{ least $k$ such that $\sum_{i=N}^k\alpha_i>x$}
\end{equation*}
which is well defined under the assumption that $\sum_{i=0}^\infty \alpha_i=\infty$. Then $m\geq r(N,x)$ is equivalent to the statement $\sum_{i=N}^m\alpha_i> x$ and so (\ref{eqn:goal0}) is equivalent to
\begin{equation*}
\sum_{i=N_0}^{n-1}\alpha_i> 2d\int_{\varepsilon_n/2d}^c\frac{dt}{\psi(t)}=2d\left(\Psi(c)-\Psi\left(\frac{\varepsilon_n}{2d}\right)\right)
\end{equation*}
for $N_0:=N\left(\frac{1}{2d}\min\left\{\psi\left(\frac{\varepsilon_n}{2d}\right),\frac{\varepsilon_n}{\alpha}\right\}\right)$. This can be reformulated as
\begin{equation}
\label{eqn:goal1}
2d\cdot \Psi\left(\frac{\varepsilon_n}{2d}\right)-\sum_{i=0}^{N_0-1}\alpha_i> 2d\Psi(c)-\sum_{i=0}^{n-1}\alpha_i
\end{equation}
All that remains is to establish (\ref{eqn:goal1}), and for this we observe that
\begin{equation*}
\begin{aligned}
2d\Psi(c)-\sum_{i=0}^{n-1}\alpha_i&<2d\Psi(c)-\sum_{i=0}^{n-2}\alpha_i\\
&=F(\varepsilon_n)\\
&=2d\cdot \Psi\left(\frac{\varepsilon_n}{2d}\right)-\alpha\cdot \tilde{N}\left(\frac{1}{2d}\min\left\{\psi\left(\frac{\varepsilon_n}{2d}\right),\frac{\varepsilon_n}{\alpha}\right\}\right)\\
&\leq 2d\cdot \Psi\left(\frac{\varepsilon_n}{2d}\right)-\alpha N_0\\
&\leq 2d\cdot \Psi\left(\frac{\varepsilon_n}{2d}\right)-\sum_{i=0}^{N_0-1}\alpha_i
\end{aligned}
\end{equation*}
\end{proof}

\subsection{Weakly contractive mappings: A simple case study}

We conclude this section by demonstrating that a quantitative convergence result for weakly contractive mappings already established in \cite{alber-guerredelabriere:97:weaklycontractive} falls out as a very simple case of our framework. In subsequent sections we will then consider more complex convergence results for asymptotically contractive mappings, and in those cases will provide new and general rates of convergence.
\begin{thm}[Cf. Theorem 3.1 of \cite{alber-guerredelabriere:97:weaklycontractive}]
\label{res:simple}
Let $E\subseteq X$ and suppose that $T:E\to X$ is a mapping which satisfies
\begin{equation*}
\norm{Tx-Ty}\leq \norm{x-y}-\psi(\norm{x-y})
\end{equation*}
for some nondecreasing function $\psi:[0,\infty)\to [0,\infty)$ which is positive on $(0,\infty)$. Let $q\in E$ be a fixpoint of $T$, and suppose that the sequence $\seq{x_n}$ in $X$ satisfies $x_{n+1}=Tx_n$ for all $n\in\NN$. Then $\norm{x_n-q}\to 0$ with rate of convergence
\begin{equation*}
\Phi(\varepsilon):=\Bigl\lceil 2\int_{\varepsilon/2}^{\norm{x_0-q}}\frac{dt}{\psi(t)}\Bigr\rceil+1
\end{equation*}
or alternatively,
\begin{equation*}
\norm{x_n-q}\leq 2\Psi^{-1}\left(\Psi(\norm{x_0-q})-\frac{n-1}{2}\right)
\end{equation*} 
for $\Psi(x):=\int^x dt/\psi(t)$. 
\end{thm}

\begin{proof}
We observe that
\begin{equation*}
\norm{x_{n+1}-q}=\norm{Tx_n-Tq}\leq \norm{x_n-q}-\psi(\norm{x_n-q})
\end{equation*}
and therefore the sequence $\seq{\norm{x_n-q}}$ satisfies (\ref{eqn:recineq:abs}) for $\alpha_n=1$, $N(\delta)=0$ and $r(N,x)=\lceil N+x\rceil$. The first rate of convergence then follows by noting that $\norm{x_0-q}$ is an upper bound for $\seq{\norm{x_n-q}}$ and plugging this data into Lemma \ref{res:recineq:first}. The second rate follows from Lemma \ref{res:translate} (setting $\beta_n=1$ and $d=1$), where having $N(\delta)=0$ allows us to define $F(\varepsilon):=2\Psi(\varepsilon/2)$.
\end{proof}

It is instructive to compare the precise formulation of Theorem \ref{res:simple} above to the corresponding Theorem 3.1 of \cite{alber-guerredelabriere:97:weaklycontractive}, as the differences represent important features of our approach which will apply throughout later sections. Firstly, we assume the existence of a fixpoint, and that allows us to weaken certain assumptions on $T$ down to those which are essential for establishing convergence. For example, here $\psi$ is not required to be continuous, neither must it satisfy the asymptotic property $\lim_{t\to\infty}\psi(t)=\infty$ (or alternatively that $E$ be bounded) as these are only required to establish the \emph{existence} of a fixpoint. 

Similarly, we assume that the sequence $\seq{x_n}$ satisfies $x_{n+1}=Tx_n$, rather than demanding any additional properties of $T$ which would ensure that Picard iterates can be generated from any initial point $x_0$. So here we do not require that $T(E)\subseteq E$, and later we can omit stronger assumptions on the domain, such as convexity. The crucial point here is that for each abstract convergence theorems we provide, there will be a natural setting setting in which fixpoints $q$ and the relevant approximating sequences $\seq{x_n}$ do indeed exist, but we do not concern ourselves with those details in this paper.

Other properties can be inferred from our assumptions. For instance, the existence of a nonnegative sequence $\seq{\mu_n}$ satisfying $\mu_{n+1}\leq \mu_n-\psi(\mu_n)$ necessarily implies that $\psi(0)=0$, even though this condition is not explicitly stated: After all, if $\psi(0)>0$ then by monotonicity we also have $\psi(\mu_n)\geq \psi(0)>0$. But since $\mu_n\to 0$ we have $\mu_n<\psi(0)$ for sufficiently large $n$, and therefore $\mu_{n+1}\leq \mu_n-\psi(\mu_n)\leq \mu_n-\psi(0)<0$, contradicting $\mu_{n+1}\geq 0$.

Finally, Theorem \ref{res:simple} illustrates our approach of providing ``proof theoretic'' rates of convergence in the sense of Section \ref{sec:basic:convdiv}. These rates of convergence are typically simpler, and crucially we do not need to prove in addition that the bounding function $\mu_n\leq f(n)$ satisfies $\lim_{t\to 0} f(t)=0$ to establish $\mu_n\to 0$. However, our Lemma \ref{res:translate} nevertheless allows us to translate our rates of convergence into traditional ones, and for concrete applications where explicit rates of convergence appear in literature, we can offer a direct comparison. For example, the rates given in our Theorem \ref{res:simple} match up well with those stated in \cite{alber-guerredelabriere:97:weaklycontractive}:
\begin{equation*}
\norm{x_n-q}\leq \Psi^{-1}(\Psi(\norm{x_0-q})-(n-1))
\end{equation*}

This similarity (differing only by a few constants) for simple cases suggests that our abstract quantitative results, which are both formulated and proven in a different style to the concrete convergence theorems they generalise, nevertheless provide good rates of convergence.

\section{Case study 1: Convergence of Mann iteration for asymptotically contractive mappings}
\label{sec:cs1}

We now present our first general convergence result, where we combine our abstract notion of being asymptotically contractive with the quantitative lemmas of the previous section to also provide a rate of convergence. The proof is extremely simple, but as we demonstrate, several existing theorems from the literature can be regarded as special cases of our general result. 
\begin{thm}
\label{res:mann}
Let $\seq{A_n}$ be a sequence of mappings $A_n:E\to X$ and $\psi:[0,\infty)\to [0,\infty)$ be some nondecreasing function with $\psi(0)=0$. Suppose that $\seq{k_n}$ is some sequence of nonnegative reals and $\sigma:(0,\infty)\times (0,\infty)\to \NN$ a modulus such that $\seq{A_n}$ is quasi asymptotically weakly contractive with respect to $q\in X$ in the sense that
\begin{equation*}
\norm{x-q}\leq b\implies \norm{A_nx-q}\leq (1+k_n)\norm{x-q}-\psi(\norm{x-q})+\delta
\end{equation*}
for all $\delta,b>0$ and $n\geq\sigma(\delta,b)$. Suppose in addition that $\seq{x_n}$ is a sequence satisfying 
\begin{equation}
\label{eqn:mann}
x_{n+1}=(1-\alpha_n)x_n+\alpha_n A_nx_n
\end{equation}
where $\seq{\alpha_n}$ is some sequence in $(0,\alpha]$ such that $\sum_{n=0}^\infty \alpha_n=\infty$ with rate of divergence $r$ and $d>0$ is such that $\prod_{i=0}^n(1+\alpha_ik_i)\leq d$ for all $n\in\NN$. Then whenever there exists $c>0$ such that $\norm{x_n-q}\leq c$ for all $n\in\NN$, we have $\norm{x_n-q}\to 0$ with rate
\begin{equation*}
\Phi_{\psi,c,\alpha,d,r,\sigma}(\varepsilon):=r\left(\sigma\left(\frac{1}{2d}\min\left\{\psi\left(\frac{\varepsilon}{2d}\right),\frac{\varepsilon}{\alpha}\right\},c\right),2d\int_{\varepsilon/2d}^c \frac{dt}{\psi(t)} \right)+1
\end{equation*}
\end{thm}

\begin{proof}
We simply observe that for any $\delta$ and all $n\geq\sigma(\delta,c)$ we have
\begin{equation*}
\begin{aligned}
\norm{x_{n+1}-q}&\leq (1-\alpha_n)\norm{x_n-q}+\alpha_n\norm{A_nx_n-q}\\
&\leq (1-\alpha_n)\norm{x_n-q}+\alpha_n((1+k_n)\norm{x_n-q}-\psi(\norm{x_n-q})+\delta)\\
&\leq (1+\alpha_nk_n)\norm{x_n-q}-\alpha_n(\psi(\norm{x_n-q})-\delta)
\end{aligned}
\end{equation*}
and therefore Lemma \ref{res:recineq:second} applies for $\mu_n:=\norm{x_n-q}$, $\beta_n:=\alpha_nk_n$ and $N(\delta):=\sigma(\delta,c)$, which yields the given rate of convergence.
\end{proof}

\begin{remark}
\label{rem:cmax}
In the case where we have a modulus $\sigma(\delta)$ which is independent of the bound $\norm{x-q}\leq b$, we no longer require a uniform bound $\norm{x_n-q}\leq c$ in order to \emph{prove} convergence of the sequence in the first place. Rather for the rate of convergence given in Theorem \ref{res:mann} it is enough to choose any $c>0$ satisfying $\max\{\norm{x_n-q}\; : \; n\in\NN\}\leq c$. Moreover, following Remark \ref{rem:uniform:bound} such a bound for the whole sequence can even by computed directly by  defining $c:=\max_{n\leq k}\{\norm{x_n-q}, 1\}$ for
\begin{equation*}
k:=r\left(M_{\psi,\alpha},2d\int_{1/2d}^{\norm{x_{M_{\psi,\alpha}}-q}} \frac{dt}{\psi(t)} \right)+1
\end{equation*}
with $M_{\psi,\alpha,d}:=\sigma\left(\frac{1}{2d}\min\left\{\psi\left(\frac{1}{2d}\right),\frac{1}{\alpha}\right\},c\right)$, where here $k$ is obtained by setting $\varepsilon=1$ in the non-independent rate of convergence for $\norm{x_n-q}\to 0$. We would then have $\norm{x_n-q}\leq 1$ for all $n\geq k$ and thus $\norm{x_n-q}\leq c$ for all $n\in\NN$.
\end{remark}

\subsection{Approximate weakly contractive mappings}
\label{sec:cs1:alber}

As our first concrete application of the main theorem in this section, we give a computational version of Section 3 of \cite{alber-guerredelabriere:97:weaklycontractive}, which considers sequence $\seq{A_n}$ of operators which are weakly contractive in the limit.
\begin{cor}
[cf. Theorem 3.4 of \cite{alber-guerredelabriere:97:weaklycontractive}]
\label{res:cs1:alber:guerre}
Let $\seq{A_n}$ be a sequence of mappings $A_n:E\to X$ and $A$ a $\psi$-weakly contractive mapping, and suppose that $\psi_n,g:[0,\infty)\to [0,\infty)$ are functions positive on $(0,\infty)$, and $\seq{k_n}$, $\seq{\mu_n}$, $\seq{h_n}$, $\seq{\delta_n}$ and $\seq{\nu_n}$ are sequences of nonnegative reals such that for $x,y\in E$ and $t\in [0,\infty)$:
\begin{equation*}
\begin{aligned}
\norm{A_nx-A_ny}&\leq (1+k_n)\norm{x-y}-\psi_n(\norm{x-y})+\mu_n\\
\norm{A_nx-Ax}&\leq h_ng(\norm{x})+\delta_n\\
|\psi_n(t)-\psi(t)|&\leq \nu_n
\end{aligned}
\end{equation*}
and $\mu_n,h_n,\delta_n,\nu_n\to 0$ with rates of convergence $f_1,f_2,f_3$ and $f_4$ respectively. Suppose in addition that $\seq{x_n}$ satisfies $x_{n+1}=A_nx_n$. Let $q$ be a fixpoint of $A$ and suppose that $\sum_{n=0}^\infty k_n\leq d$ for some $d>0$. Then $\norm{x_n-q}\to 0$ with rate of convergence
\begin{equation*}
\begin{aligned}
&\Phi_{\psi,c,c_1,d,f_1,f_2,f_3,f_4}(\varepsilon):=\\
&\sigma_{c_1,f_1,f_2,f_3,f_4}\left(\frac{1}{2e^d}\min\left\{\psi\left(\frac{\varepsilon}{2e^d}\right),\varepsilon\right\}\right)+\Bigl\lceil 2e^d\int_{\varepsilon/2e^d}^c\frac{dt}{\psi(t)}\Bigr\rceil +1
\end{aligned}
\end{equation*}
for 
\begin{equation*}
\sigma_{c_1,f_1,f_2,f_3,f_4}(\delta):=\max\left\{f_1\left(\frac{\delta}{4}\right),f_2\left(\frac{\delta}{4c_1}\right),f_3\left(\frac{\delta}{4}\right),f_4\left(\frac{\delta}{4}\right)\right\}
\end{equation*}
and where $c,c_1>0$ are any reals satisfying $\max\{\norm{x_n-q} \; : \; n\in\NN\}\leq c$ and $g(\norm{q})\leq c_1$.
\end{cor}

\begin{proof}
If $q$ is a fixpoint of $A$ then we have
\begin{equation*}
\begin{aligned}
\norm{A_nx-q}&\leq \norm{A_nx-A_nq}+\norm{A_nq-Aq}\\
&\leq (1+k_n)\norm{x-q}-\psi_n(\norm{x-q})+\mu_n+h_ng(\norm{q})+\delta_n\\
&\leq (1+k_n)\norm{x-q}-\psi(\norm{x-q})+(\mu_n+h_ng(\norm{q})+\delta_n+\nu_n)
\end{aligned}
\end{equation*}
and therefore the $\seq{A_n}$ are quasi asymptotically weakly contractive in the sense of Theorem \ref{res:mann} with modulus $\sigma_{c_1,f_1,f_2,f_3,f_4}$ (which is also uniform in the bound $b$). The Picard sequence $\seq{x_n}$ is a special case of the scheme (\ref{eqn:mann:iteration}) with $\alpha_n=1$ for all $n\in\NN$, and using the inequality $1+x\leq e^x$ we have $\prod_{n=0}^n(1+k_i)\leq e^d$ for any $n\in\NN$. Therefore we can apply Theorem \ref{res:mann} and Remark \ref{rem:cmax} directly with $\alpha=1$ and $r(N,x):=\lceil x+N\rceil$ to obtain the given rate of convergence.
\end{proof}

\subsection{Totally asymptotically weakly contractive mappings}
\label{sec:cs1:chidume}

We now give a quantitative convergence proof relating to so-called totally asymptotically weakly contractive mappings \cite{alber-chidume-zegeye:06:nonexpansive}, a class of mapping which we have already alluded to in Section \ref{sec:basic:asymp}.
\begin{cor}
[cf. Theorem 4.1 of \cite{alber-chidume-zegeye:06:nonexpansive}]
\label{res:cs1:chidume}
Let $T:E\to X$ be a totally asymptotically weakly contractive mapping in the sense that there exist nondecreasing functions $\phi,\psi:[0,\infty)\to [0,\infty)$ with $\phi(0)=\psi(0)=0$ along with sequences $\seq{\nu_n},\seq{l_n}$ of nonnegative reals such that
\begin{equation*}
\norm{T^nx-T^ny}\leq \norm{x-y}+\nu_n\phi(\norm{x-y})-\psi(\norm{x-y})+l_n
\end{equation*}
and $\nu_n,l_n\to 0$ with rates $f_1$ and $f_2$ respectively. Suppose in addition that $\seq{x_n}$ is a sequence satisfying 
\begin{equation*}
x_{n+1}=(1-\alpha_n)x_n+\alpha_nT^nx_n
\end{equation*}
where $\seq{\alpha_n}$ is some sequence in $(0,\alpha]$ such that $\sum_{n=0}^\infty \alpha_n=\infty$ with rate of divergence $r$. Let $q$ be a fixpoint of $T$. Then whenever there exists $c>0$ such that $\norm{x_n-q}\leq c$ for all $n\in\NN$, we have $\norm{x_n-q}\to 0$ with rate
\begin{equation*}
\Phi_{\psi,\phi,c,\alpha,r,f_1,f_2}:=r\left(\sigma_{f_1,f_2,\phi}\left(\frac{1}{2}\min\left\{\psi\left(\frac{\varepsilon}{2}\right),\frac{\varepsilon}{\alpha}\right\},c\right),2\int_{\varepsilon/2}^c \frac{dt}{\psi(t)} \right)+1
\end{equation*}
for
\begin{equation*}
\sigma_{f_1,f_2,\phi}(\delta,b):=\max\left\{f_1\left(\frac{\delta}{2\phi(b)}\right),f_2\left(\frac{\delta}{2}\right)\right\}
\end{equation*}
\end{cor}
 
\begin{proof}
If $q$ is a fixpoint of $T$ then whenever $\norm{x-q}\leq b$ we have
\begin{equation*}
\begin{aligned}
\norm{T^nx-q}&\leq \norm{x-q}+\nu_n\phi(\norm{x-q})-\psi(\norm{x-q})+l_n\\
&\leq \norm{x-q}-\psi(\norm{x-q})+(\nu_n\phi(b)+l_n)
\end{aligned}
\end{equation*}
and therefore the sequence $\seq{T^n}$ is quasi asymptotically weakly contractive in the sense of Theorem \ref{res:mann} (for $k_n=0$) with modulus $\sigma_{f_1,f_2,\phi}$. Therefore Theorem \ref{res:mann} applies directly with $d:=1$ and results in the given rate of convergence.
\end{proof} 
  
\begin{remark}
For the special case of totally asymptotically weakly contractive mappings where $\phi$ is a linear function i.e. $\phi(t)=at$ and $l_n=0$, and subject to the condition $\sum_{n=0}^\infty\alpha_n\nu_n<\infty$, the rate of convergence we obtain is particularly simple. Here the main recursive inequality in the proof of Theorem \ref{res:mann} reduces to
\begin{equation*}
\norm{x_{n+1}-q}\leq (1+\alpha_n\nu_n)\norm{x_n-q}-\alpha_n\psi(\norm{x_n-q})
\end{equation*}
and so a rate of convergence for $\norm{x_n-q}\to 0$ is given by
\begin{equation*}
\Phi_{\psi,c,\alpha,d,r}(\varepsilon):=r\left(0,2e^d\int_{\varepsilon/2e^d}^c\frac{dt}{\psi(t)}\right)+1
\end{equation*}
where $\sum_{i=0}^n \alpha_ n\nu_n\leq d$ for all $n\in\NN$. Our connection with traditional rates of convergence in Lemma \ref{res:translate} gives, in this case, the alternative rate of convergence
\begin{equation*}
\norm{x_n-q}\leq 2e^d\Psi^{-1}\left(\Psi(\norm{x_0-q}-\frac{1}{2e^d}\sum_{i=0}^{n-2} \alpha_i\right)
\end{equation*}
which is broadly similar to the explicit rate of convergence provided for this case in \cite[p. 10--11]{alber-chidume-zegeye:06:nonexpansive}. Further generalisations in this direction are possible, for cases where $\phi(t)$ is linear or eventually linear for sufficiently large $t$, and rates of convergence can be provided for other results in Section 4 of \cite{alber-chidume-zegeye:06:nonexpansive}.
\end{remark}

\section{Case study 2: Asymptotically $d$-weakly contractive mappings in spaces with a uniformly continuous duality selection map}
\label{sec:cs2}

We now move to the setting of uniformly smooth Banach spaces, and turn our attention to a variant of weak contractivity which involves the dual space. These are particularly interesting for us, as the associated convergence results often rely on geometric properties of the underlying space, such a uniform smoothness, allowing us to produce rates of convergence in terms of the corresponding moduli. 

We start with some standard facts. Let $X^\ast$ be the dual of the space $X$, and $J:X\to 2^{X^\ast}$ the normalized duality mapping defined by
\begin{equation*}
Jx:=\{j\in X^\ast\; : \; \pair{x,j}=\norm{x}^2=\norm{j}^2\}
\end{equation*}
where $\pair{\; ,\; }$ denotes the duality pairing. We will make use of the following well known inequality: For $x,y\in X$ and $j\in J(x+y)$ we have
\begin{equation}
\label{eqn:geom}
\norm{x+y}^2\leq \norm{x}^2+2\pair{y,j}.
\end{equation}
\begin{defn}[\cite{alber-guerredelabriere:01:projection}]
\label{defn:dweakly}
A mapping $T:E\to X$ is called $d$-weakly contractive relative to some continuous and strictly increasing function $\psi:[0,\infty)\to [0,\infty)$ positive on $(0,\infty)$ and satisfying $\psi(0)=0$ and $\lim_{t\to\infty}\psi(t)=\infty$ if for all $x,y\in E$ there exists some $j\in J(x-y)$ such that
\begin{equation*}
|\pair{Tx-Ty,j}|\leq \norm{x-y}^2-\psi(\norm{x-y})
\end{equation*}
\end{defn}
In order to establish convergence to fixpoints for mappings of this kind, we will appeal to an alternative characterisation of a uniformly smooth space as being one equipped with a norm-to-norm uniformly continuous duality selection map:
\begin{defn}[\cite{kohlenbach-leustean:12:modulus}]
\label{defn:selectionmap}
A space with a uniformly continuous duality selection map $(X,J)$ is defined to be any Banach space $X$ equipped with a mapping $J:X\to X^\ast$ which satisfies
\begin{enumerate}

\item $\pair{x,Jx}=\norm{x}^2=\norm{Jx}^2$ for all $x\in X$,

\item $J$ is norm-to-norm uniformly continuous on bounded subsets of $X$.

\end{enumerate}
Furthermore, a modulus of uniform continuous for $J$ is defined to be any function $\omega:(0,\infty)\times (0,\infty)\to (0,\infty)$ such that for all $x,y\in X$ with $\norm{x},\norm{y}\leq d$:
\begin{equation*}
\norm{x-y}_X\leq \omega(d,\varepsilon)\implies \norm{Jx-Jy}_{X^\ast}\leq\varepsilon
\end{equation*}

\end{defn}
If $X$ is smooth then the normalised duality mapping is single-valued, and so coincides with a selection map in the sense of Definition \ref{defn:selectionmap}. Moreover, if $X$ is \emph{uniformly} smooth, then $J$ is uniformly continuous on bounded subsets, and therefore the uniformly smooth spaces form a natural class of spaces which always possess a uniformly continuous duality selection map. That the converse also holds, namely that any Banach space possessing a uniformly continuous duality selection map is uniformly smooth, is proven in \cite[Appendix A]{koernlein:15:halpern}. 
\begin{thm}
\label{res:dweakly}
Suppose that $(X,J)$ is a space equipped with a uniformly continuous duality selection map, with modulus of continuity $\omega$. Let $\seq{A_n}$ be a sequence of mappings $A_n:E\to X$ and $\psi:[0,\infty)\to [0,\infty)$ be some nondecreasing function with $\psi(0)=0$. Suppose that $\seq{A_n}$ is quasi asymptotically $d$-weakly contractive w.r.t. some $q\in X$ and with modulus $\sigma:(0,\infty)\times (0,\infty)\to \NN$, in the sense that for any $\delta,b>0$ and $x\in E$ we have
\begin{equation*}
\norm{x-q}\leq b\implies |\pair{A_nx-q,J(x-q)}|\leq\norm{x-q}^2-\psi(\norm{x-q})+\delta
\end{equation*}
for all $n\geq\sigma(\delta,b)$. Suppose in addition that $\seq{x_n}$ is a sequence satisfying (\ref{eqn:mann}) where $\seq{\alpha_n}$ is some sequence in $(0,\alpha]$ such that $\alpha_n\to 0$ with rate of convergence $f$ and $\sum_{n=0}^\infty \alpha_n=\infty$ with rate of divergence $r$. Then whenever $c_1,c_2>0$ are such that $\norm{x_n-q}\leq c_1$ and $\norm{A_nx_n-x_n}\leq c_2$ for all $n\in\NN$, we have $\norm{x_n-q}\to 0$ with rate
\begin{equation*}
\Phi_{\omega,\psi,c_1,c_2,\alpha,f,r,\sigma}(\varepsilon):=
r\left(N_{\omega,c_1,c_2,f,\sigma}\left(\frac{1}{2}\min\left\{2\psi\left(\frac{\varepsilon}{\sqrt{2}}\right),\frac{\varepsilon^2}{\alpha}\right\}\right),2\int_{\varepsilon^2/2}^{c_1} \frac{dt}{2\psi(\sqrt{t})} \right)
\end{equation*}
where
\begin{equation*}
N_{\omega,c_1,c_2,f,\sigma}(\delta):=\max\left\{\sigma\left(\frac{\delta}{4},c_1\right),f\left(\frac{1}{c_2}\cdot \omega\left(c_1,\frac{\delta}{4c_2}\right)\right)   \right\}
\end{equation*}

\end{thm}

\begin{proof}
We start by using the property of $\seq{A_n}$ and observing that for any $\delta>0$, since $\norm{x_n-q}\leq c_1$ we have 
\begin{equation*}
\begin{aligned}
\pair{A_nx_n-q,J(x_n-q)}&\leq \norm{x_n-q}^2-\psi(\norm{x_n-q})+\frac{\delta}{4}\\
&=\pair{x_n-q,J(x_n-q)}-\psi(\norm{x_n-q})+\frac{\delta}{4}
\end{aligned}
\end{equation*}
and therefore 
\begin{equation}
\label{eqn:dweakly:0}
\pair{Ax_n-x_n,J(x_n-q)}\leq -\psi(\norm{x_n-q})+\frac{\delta}{4}
\end{equation}
$n\geq\sigma(\delta/4,c_1)$. Independently of this, we observe that whenever
\begin{equation}
\label{eqn:dweakly:1}
\norm{J(x_{n+1}-q)-J(x_n-q)}\leq \frac{\delta}{4c_2}
\end{equation}
it follows that
\begin{equation}
\label{eqn:dweakly:2}
\begin{aligned}
&\pair{Ax_n-x_n,J(x_{n+1}-q)-J(x_n-q)}\\
&\leq \norm{Ax_n-x_n}\cdot \norm{J(x_{n+1}-q)-J(x_n-q)}\\
&\leq c_2\norm{J(x_{n+1}-q)-J(x_n-q)}\leq \frac{\delta}{4}
\end{aligned}
\end{equation}
Now, noting that $\norm{x_n-q},\norm{x_{n+1}-q}\leq c_1$, by uniform continuity of $J$ we have that (\ref{eqn:dweakly:1}) holds as long as
\begin{equation}
\label{eqn:dweakly:3}
\norm{x_{n+1}-x_n}\leq\omega\left(c_1,\frac{\delta}{4c_2}\right)
\end{equation}
Observing that
\begin{equation*}
\norm{x_{n+1}-x_n}=\norm{\alpha_n(A_nx_n-x_n)}\leq \alpha_nc_2
\end{equation*}
and recalling that $\alpha_n\to 0$ with rate $f$, it follows that (\ref{eqn:dweakly:3}) and therefore (\ref{eqn:dweakly:1}) and (\ref{eqn:dweakly:2}) hold whenever
\begin{equation*}
n\geq f\left(\frac{1}{c_2}\cdot \omega\left(c_1,\frac{\delta}{4c_2}\right)\right)
\end{equation*}
Putting together (\ref{eqn:dweakly:0}) and (\ref{eqn:dweakly:2}) we see that
\begin{equation}
\label{eqn:dweakly:4}
\begin{aligned}
&\pair{A_nx_n-x_n,J(x_{n+1}-q)}\\
&\leq \pair{A_nx_n-x_n,J(x_n-q)}+\pair{A_nx_n-x_n,J(x_{n+1}-q)-J(x_n-q)}\\
&\leq -\psi(\norm{x_n-q})+\frac{\delta}{2}
\end{aligned}
\end{equation}
for all $n\geq N_{\omega,c_1,c_2,f,\sigma}(\delta)$ where the latter is defined as in the statement of the theorem. Thus for all $n$ in this range we have, using (\ref{eqn:dweakly:4}) together with (\ref{eqn:geom}):
\begin{equation}
\label{eqn:dweakly:rec}
\begin{aligned}
\norm{x_{n+1}-q}^2&=\norm{x_n-q+\alpha_n(A_nx_n-x_n)}\\
&\leq \norm{x_n-q}^2+2\alpha_n\pair{A_nx_n-x_n,J(x_{n+1}-q)}\\
&\leq\norm{x_n-q}^2-2\alpha_n\psi(\norm{x_n-q})+\alpha_n\delta
\end{aligned}
\end{equation}
and therefore $\mu_n:=\norm{x_n-q}^2$ satisfies
\begin{equation*}
\mu_{n+1}\leq \mu_n-\alpha_n\bar{\psi}(\mu_n)+\alpha_n\delta
\end{equation*}
for $\bar{\psi}(t):=2\psi(\sqrt{t})$. Applying Lemma \ref{res:recineq:first} to $\seq{\mu_n}$ on parameters $\bar{\psi}$, $\alpha$, $r$, $N_{\omega,c_1,c_2,f,\sigma}$ and $c_1$ gives the following rate of convergence for $\norm{x_n-q}^2\to 0$:
\begin{equation*}
\varepsilon\mapsto r\left(N_{\omega,c_1,c_2,f,\sigma}\left(\frac{1}{2}\min\left\{\bar\psi\left(\frac{\varepsilon}{2}\right),\frac{\varepsilon}{\alpha}\right\}\right),2\int_{\varepsilon/2}^{c_1} \frac{dt}{\bar\psi(t)} \right)
\end{equation*}
We can adjust this to a rate of convergence for $\norm{x_n-q}\to 0$ via the substitution $\varepsilon\mapsto \varepsilon^2$, and writing out the definition of $\bar\psi$ in full we obtain the rate stated in the theorem.
%
%
\end{proof}

\begin{remark}
\label{rem:retraction}
In both Theorem \ref{res:dweakly}, we can easily replace the basic Mann scheme (\ref{eqn:mann}) with the following variant:
\begin{equation*}
x_{n+1}=P((1-\alpha_n)x_n+\alpha_nA_nx_n)
\end{equation*}
where $P:X\to E$ is a nonexpansive retraction, that is, a mapping satisfying $P^2=P$ and $\norm{Px-Py}\leq \norm{x-y}$ for all $x,y\in X$. We can do this by simply setting $y_n:=(1-\alpha_n)x_n+\alpha_nA_nx_n$ and then replacing $x_{n+1}$ with $y_n$ throughout, and finally for (\ref{eqn:dweakly:rec}) observing that
\begin{equation*}
\begin{aligned}
\norm{x_{n+1}- q}^2&=\norm{P((1-\alpha_n)x_n+\alpha_nA_nx_n)-Pq}\\
&\leq \norm{x_n-q+\alpha_n(A_nx_n-x_n)}\\
&\leq \norm{x_n-q}^2-2\alpha_n\pair{A_nx_n-x_n,J(y_{n+1}-q)}
\end{aligned}
\end{equation*}

\end{remark}

\begin{remark}
\label{rem:unifsmooth}
It has been shown in \cite{kohlenbach-leustean:12:modulus} that in the case that $X$ is uniformly smooth, a modulus of uniform continuity for $J$ can be constructed in terms of a modulus $\tau$ of uniform smoothness for $X$. More specifically, we define
\begin{equation*}
\omega_\tau(d,\varepsilon):=\frac{\varepsilon^2}{12d}\cdot \tau\left(\frac{\varepsilon}{2d}\right), \ \ \ \varepsilon\in (0,2], d\geq 1
\end{equation*}
with $\omega_\tau(d,\varepsilon):=\omega_\tau(1,\varepsilon)$ for $d<1$ and $\omega_\tau(d,\varepsilon):=\omega_\tau(d,2)$ for $\varepsilon>2$. Thus the requirement in Theorem \ref{res:dweakly} to provide an explicit modulus of uniform continuity for $J$ can be replaced by instead providing a modulus of uniform smoothness for $X$.
\end{remark}

\subsection{$d$-weakly contractive mappings uniformly smooth spaces}
\label{sec:cs2:concrete}

It is clear that if $T:E\to X$ is a $d$-weakly contractive mapping in the sense of Definition \ref{defn:dweakly} then $\seq{T}$ is quasi asymptotically weakly contractive relative to any fixpoint $q$ of $T$ the sense of Theorem \ref{res:dweakly} with modulus $\sigma(\delta,b):=0$. In fact, we can give the following as a simple corollary, which forms a computational version of Theorem 3.1 of Chidume et al \cite{chidume-zegeye-aneke:02:dweakly}:
\begin{cor}
[Cf. Theorem 3.1 of \cite{chidume-zegeye-aneke:02:dweakly}]
\label{res:dweakly:chidume}
Suppose that $X$ is a uniformly smooth space equipped with a modulus $\tau$. Let $T:E\to X$ be a $d$-weakly contractive mapping w.r.t. $\psi$ (now in the sense of Definition \ref{defn:dweakly}) and suppose in addition that $\seq{x_n}$ is a sequence satisfying 
\begin{equation*}
x_{n+1}=P((1-\alpha_n)x_n+\alpha_nTx_n)
\end{equation*}
where $P:X\to E$ is some nonexpansive retraction, $\seq{\alpha_n}$ is some sequence in $(0,\alpha]$ such that $\alpha_n\to 0$ with rate of convergence $f$ and $\sum_{n=0}^\infty \alpha_n=\infty$ with rate of divergence $r$. Let $q\in E$ be a fixpoint of $T$. Then whenever $c_1,c_2>0$ are such that $\norm{x_n-q}\leq c_1$ and $\norm{Tx_n-x_n}\leq c_2$ for all $n\in\NN$, we have $\norm{x_n-q}\to 0$ with rate
\begin{equation*}
\Phi_{\tau,\psi,c_1,c_2,\alpha,f,r}(\varepsilon):=
r\left(N_{\tau,c_1,c_2,f}\left(\frac{1}{2}\min\left\{2\psi\left(\frac{\varepsilon}{\sqrt{2}}\right),\frac{\varepsilon^2}{\alpha}\right\}\right),2\int_{\varepsilon^2/2}^{c_1} \frac{dt}{2\psi(\sqrt{t})} \right)
\end{equation*}
where
\begin{equation*}
N_{\tau,c_1,c_2,f}(\delta):=f\left(\frac{1}{c_2}\cdot \omega_\tau\left(c_1,\frac{\delta}{4c_2}\right)\right)   
\end{equation*}
for $\omega_\tau$ defined as in Remark \ref{rem:unifsmooth}.
\end{cor}

\begin{proof}
This follows directly from Theorem \ref{res:dweakly} and Remark \ref{rem:retraction}, observing that for $q$ a fixpoint of $T$ the constant sequence $\seq{T}$ is quasi asymptotically $d$-weakly contractive with modulus $\sigma(\delta,b):=0$.
\end{proof}
\begin{remark}
Corollary \ref{res:dweakly:chidume} includes an implicit assumption that the sequences $\seq{\norm{x_n-q}}$ and $\seq{\norm{Tx_n-x_n}}$ are bounded, whereas the proof of Theorem 3.1 of \cite{chidume-zegeye-aneke:02:dweakly} demonstrates that these sequences are always bounded, and so $c_1,c_2$ are guaranteed to exist in this case. However, the rate of convergence we provide remains valid, and is in any case dependent on upper bounds for the sequences, independent of whether or not their existence is required in advance to establish convergence.
\end{remark}

\section{Case study 3: Perturbed Mann schemes}
\label{sec:cs3}

In our final and most complex case study, we take as inspiration a paper of Alber, Reich and Yao \cite{alber-reich-yao:03:weaklycontractive}, where approximation sequences $\seq{z_n}$ to weakly contractive mappings $T:E\to X$ are studied for which $z_n$ is projected onto some $E_n\subseteq E$. To establish convergence we require that the perturbed sets $\seq{E_n}$ approach $E$ in the uniform Hausdorff metric, and that the projection operators are sunny nonexpansive retractions. We start off by introducing these notions and providing some key computational lemmas that will be needed to establish rates of convergence.

The Hausdorff distance between two (non-empty) subsets of $X$ is defined by
\begin{equation*}
H(P,Q):=\max\left\{\sup_{x\in P}\inf_{y\in Q}\norm{x-y},\sup_{y\in Q}\inf_{x\in P}\norm{x-y}\right\}
\end{equation*}  
Following \cite{kohlenbach-powell:20:accretive}, we represent the Hausdorff distance via an abstract predicate $H^\ast$, defined as follows:
\begin{defn}
\label{defn:hausdorffpred}
For $P,Q\subseteq X$ and $a>0$ say that $H^\ast[P,Q,a]$ is true iff both
\begin{equation*}
\forall x\in P\; \exists y\in Q(\norm{x-y}\leq a)\mbox{ \ \ and \ \  }\forall y\in Q\; \exists x\in P(\norm{x-y}\leq a)
\end{equation*}
\end{defn}
\begin{lem}
\label{res:hausdorff:comp}
If $H(P,Q)<a$ then $H^\ast[P,Q,a]$ holds.
\end{lem}

\begin{proof}
From $\sup_{x\in P}\inf_{y\in Q}\norm{x-y}<a$ we infer that for all $x\in P$:
\begin{equation*}
\inf_{y\in Q}\norm{x-y}<a
\end{equation*}
Supposing for contradiction that for all $y\in Q$ we had $\norm{x-y}\geq a$, it would follow that $\inf_{y\in Q}\norm{x-y}\geq a$, so therefore for each $x\in P$ there must exists at least one $y\in Q$ with $\norm{x-y}<a$. Showing that for any $y\in Q$ there exists some $x\in P$ with $\norm{x-y}<a$ is entirely analogous, and therefore $H^\ast[P,Q,a]$ holds.
\end{proof}

Lemma \ref{res:hausdorff:comp} is useful as it shows us that we can replace assumptions involving the Hausdorff metric with assumptions phrased in terms of the simpler predicate $H^\ast$ (which does not involve infima or suprema, nor indeed any requirement that the Hausdorff metric be well-defined). We now turn to the notion of a sunny nonexpansive retraction (\cite{bruck:73:nonexpansive:projections,goebel-reich:84:book}) -- which has been recently studied from the perspective of proof mining in \cite{kohlenbach-sipos:21:sunny} -- and give a quantitative version of a key lemma from \cite{alber-reich-yao:03:weaklycontractive} in which we make use of our Hausdorff predicate. We first require a definition which gives a characterisation of a sunny expansiveness which will be crucial in what follows.
\begin{defn}[cf. \cite{alber-reich-yao:03:weaklycontractive}]
\label{defn:retractions}
Let $E$ be a nonempty, closed convex subset of a Banach space $X$. A nonexpansive retraction $Q:X\to E$ is called sunny if for all $x\in X$ and $t\geq 0$ we have
\begin{equation*}
Q(Qx+t(x-Qx))=Qx
\end{equation*}
Moreover, if $X$ is smooth, $Q$ is sunny nonexpansive iff for all $x\in X$ and $y\in E$:
\begin{equation}
\label{eqn:sunny}
\pair{x-Qx,J(y-Qx)}\leq 0
\end{equation}
\end{defn}
In the remainder of this section we will work in a space $(X,J)$ with a uniformly continuous duality selection map, and \emph{define} a nonexpansive retraction $Q:X\to E$ to be one which satisfies (\ref{eqn:sunny}) for all $x\in X$ and $y\in E$ (note that the assumption that $E$ is closed and convex is no longer used, though in the case where it is our definition will then match up to the usual definition.

With all this now in place, we now give a abstract formulation of Lemma 3.4 of \cite{alber-reich-yao:03:weaklycontractive} in terms of our moduli:
\begin{lem}
\label{lem:sunnyhaus}
Let $(X,J)$ be a space with a uniformly continuous duality selection map with modulus of continuity $\omega$, and $Q_1:X\to E_1$, $Q_2:X\to E_2$ be sunny nonexpansive retractions with $\norm{Q_10},\norm{Q_20}\leq d$. Pick any $\varepsilon,b>0$ and define 
\begin{equation*}
R:=2(2b+d)+1 \mbox{ \ \ and \ \ } a:=\min\left\{1,\omega\left(R,\frac{\varepsilon}{R}\right)\right\}
\end{equation*}
Then $H^\ast[E_1,E_2,a]$ implies that for any $x\in X$ with $\norm{x}\leq b$ we have 
\begin{equation*}
\norm{Q_1x-Q_2x}^2\leq\varepsilon
\end{equation*}
\end{lem}

\begin{proof}
Take $x\in X$ with $\norm{x}\leq b$ and let $y_1:=Q_1x\in E_1$ and $y_2:=Q_2x\in E_2$. Then by $H^\ast[E_1,E_2,a]$ there exists some $v_2\in E_2$ such that $\norm{y_1-v_2}\leq a$. We have
\begin{equation}
\label{eqn:sunnyhaus0}
\begin{aligned}
&\pair{x-y_2,J(y_1-y_2)}\\
&= \pair{x-y_2,J(v_2-y_2)}+\pair{x-y_2,J(y_1-y_2)-J(v_2-y_2)}\\
&\leq \norm{x-y_2}\norm{J(y_1-y_2)-J(v_2-y_2)}
\end{aligned}
\end{equation}
where for the second inequality we use $\pair{x-y_2,J(v_2-y_2)}\leq 0$ which follows directly from (\ref{eqn:sunny}) and the assumption that $Q_2$ is sunny. We now observe that
\begin{equation*}
\norm{x-y_2}\leq \norm{x-Q_20}+\norm{Q_20-Q_2x}\leq 2\norm{x}+\norm{Q_20}\leq 2b+d
\end{equation*}
and similarly $\norm{x-y_1}\leq 2b+d$. From this it follows that $\norm{y_1-y_2}\leq 2(2b+d)$ and, using $\norm{y_1-v_2}\leq a\leq 1$, we have
\begin{equation*}
\norm{v_2-y_2}\leq \norm{v_2-y_1}+{y_1-y_2}\leq a+2(2b+d)\leq R
\end{equation*}
Using now that $\norm{y_1-y_2},\norm{v_2-y_2}\leq R$ and $\norm{y_1-v_2}\leq a\leq \omega\left(R,\frac{\varepsilon}{R}\right)$, since $\omega$ is a modulus of continuity for $J$ in the sense of Definition \ref{defn:selectionmap} we have
\begin{equation*}
\norm{J(y_1-y_2)-J(v_2-y_2)}\leq \frac{\varepsilon}{R}
\end{equation*}
and therefore from (\ref{eqn:sunnyhaus0}):
\begin{equation*}
\pair{x-y_2,J(y_1-y_2)}\leq (2b+d)\cdot \frac{\varepsilon}{R}\leq \frac{\varepsilon}{2}
\end{equation*}
Entirely analogously, by $H^\ast[E_1,E_2,a]$ there exists some $v_1\in E_1$ such that $\norm{v_1-y_2}\leq a$ and using a symmetric argument we can show that
\begin{equation*}
\pair{x-y_1,J(y_2-y_1)}\leq\frac{\varepsilon}{2}
\end{equation*}
and therefore
\begin{equation*}
\begin{aligned}
\norm{y_1-y_2}^2&=\pair{y_1-y_2,J(y_1-y_2)}\\
&=\pair{y_1-x,J(y_1-y_2)}+\pair{x-y_2,J(y_1-y_2)}\leq\varepsilon
\end{aligned}
\end{equation*}
This completes the proof.
\end{proof}
\begin{remark}
The above lemma essentially provides a \emph{modulus of uniqueness} for the implication $H(E_1,E_2)=0\implies Q_1=Q_2$, in the sense that it tells us exactly how close $E_1$ and $E_2$ need to be in order to make $\norm{Q_1x-Q_2x}$ arbitrarily small. The extraction of moduli of uniqueness form an important subclass of applications in proof mining, see \cite[Chapters 15-16]{kohlenbach:08:book} for examples.
\end{remark}

We are now ready to present the main quantitative result:
\begin{thm}
\label{res:perturbed}
Suppose that $(X,J)$ is a space equipped with a uniformly continuous duality selection map, with modulus of continuity $\omega$. Let $\seq{A_n}$ be a sequence of mappings $A_n:E_n\to X$ and $\psi:[0,\infty)\to [0,\infty)$ be a nondecreasing function with $\psi(0)=0$. Suppose that $\seq{A_n}$ are asymptotically $\psi$-weakly contractive with modulus $\sigma$ in the sense that for all $x,y\in X$ we have
\begin{equation*}
\norm{x-y}\leq b\implies \norm{A_nx-A_ny}\leq \norm{x-y}-\psi(\norm{x-y})+\delta
\end{equation*}
for all $\delta,b>0$ and $n\geq \sigma(\delta,b)$. Suppose in addition that $\seq{z_n}$ is a sequence satisfying
\begin{equation*}
\label{eqn:mann:perturbed}
z_{n+1}=Q_n((1-\alpha_n)z_n+\alpha_nA_nz_n)
\end{equation*}
where $\seq{\alpha_n}$ is some sequence in $(0,\alpha]$ such that $\sum_{n=0}^\infty\alpha_n=\infty$ with rate $r$, and $Q_n:X\to E_n$ is some sequence of sunny nonexpansive retractions. Let $Q:X\to E$ be a sunny nonexpansive retraction and suppose that $H^\ast[E_n,E,a_n]$ holds for all $n\in\NN$, where $\seq{a_n}$ is a sequence in $(0,1)$. Suppose that $d>0$ is such that $Q0\leq d$ and $Q_n0\leq d$ for all $n\in\NN$. Fix some $q\in E$ such that $\norm{A_nq-q}\to 0$ with rate $f$. Finally, let $\seq{x_n}$ be any non-perturbed sequence defined by
\begin{equation*}
x_{n+1}=Q((1-\alpha_n)x_n+\alpha_nA_nx_n)
\end{equation*} 
and suppose that $\seq{z_n}$ and $\seq{x_n}$ are bounded. Then $\seq{\norm{A_nz_n}}$ is also bounded, and we choose $c_1,c_2,c_3,c_4>0$ be such that $\norm{z_n}\leq c_1$, $\norm{A_nz_n}\leq c_2$, $\norm{x_n-q}\leq c_3$, $\norm{x_n-z_n}\leq c_4$ for all $n\in\NN$. Let $R:=2(2(c_1+\alpha c_2)+d)+1$. Then whenever $h:(0,\infty)\to \NN$ is such that for any $\delta>0$ we have
\begin{equation}
\label{eqn:aconv}
a_n\leq \omega\left(R,\frac{(\alpha_n\delta)^2}{4R}\right)
\end{equation}
for all $n\geq h(\delta)$, then $\norm{z_n-q}\to 0$ with rate of convergence

\begin{equation*}
\begin{aligned}
&\Phi_{\psi,c_3,c_4,c_5,\alpha,r,\sigma,f,h}(\varepsilon):=\\
&r\left( N_{c_3,c_4,\sigma,f,h}\left(\frac{1}{2}\min\left\{\psi\left(\frac{\varepsilon}{4}\right),\frac{\varepsilon}{2\alpha}\right\}\right),2\int_{\varepsilon/4}^{\max\{c_3,c_4\}} \frac{dt}{\psi(t)} \right)+1
\end{aligned}
\end{equation*}
where
\begin{equation*}
N_{c_3,c_4,\sigma,f,h}(\delta):=\max\left\{\sigma\left(\frac{\delta}{2},\max\{c_3,c_4\}\right),f\left(\frac{\delta}{2}\right),h(\delta)\right\}
\end{equation*}

%
%
%
%
%
%
\end{thm}

\begin{proof}
To show that $\seq{\norm{A_nz_n}}$ is bounded it is enough to observe that for sufficiently large $n$ we have
\begin{equation*}
\begin{aligned}
\norm{A_nz_n}&\leq \norm{A_nz_n-A_nq}+\norm{A_nq-q}+\norm{q}\\
&\leq \norm{z_n-q}-\psi(\norm{z_n-q})+1+\norm{A_nq-q}+\norm{q}
\end{aligned}
\end{equation*}
and so boundedness of $\seq{\norm{A_nz_n}}$ follows from boundedness of $\seq{\norm{z_n}}$ and $\norm{A_nq-q}\to 0$. We now start off the main proof by showing that $\norm{x_n-q}\to 0$, which is established in the standard way: Since $q\in E$ and $Q$ is a nonexpansive retraction, we observe that for any $\delta>0$ we have
\begin{equation*}
\begin{aligned}
\norm{x_{n+1}-q}&=\norm{Q((1-\alpha_n)x_n+\alpha_nA_nx_n)-Qq}\\
&\leq \norm{(1-\alpha_n)x_n+\alpha_nA_nx_n-q}\\
&\leq (1-\alpha_n)\norm{x_n-q}+\alpha_n\norm{A_nx_n-q}\\
&\leq (1-\alpha_n)\norm{x_n-q}+\alpha_n\norm{A_nx_n-A_nq}+\alpha_n\norm{A_nq-q}\\
&\leq \norm{x_n-q}-\alpha_n\psi(\norm{x_n-q})+\alpha_n\frac{\delta}{2}+\alpha_n\frac{\delta}{2}
\end{aligned}
\end{equation*}
where for the last step we require that $n\geq \sigma(\frac{\delta}{2},c_3)$ (from asymptotic contractivity of $\seq{A_n}$ with modulus $\sigma$) and $n\geq f(\frac{\delta}{2})$ (from $\norm{A_nq-q}\to 0$ with rate $f$). In other words, this holds for all $n\geq N^1_{c_3,\sigma,f}(\delta)$ where
\begin{equation*}
N^1_{c_3,\sigma,f}(\delta):=\max\left\{\sigma\left(\frac{\delta}{2},c_3\right),f\left(\frac{\delta}{2}\right)\right\}
\end{equation*}
Therefore by Lemma \ref{res:recineq:first} applied to $\mu_n:=\norm{x_n-q}$ and parameters $\psi$, $\alpha$, $r$,$N^1_{c_3,\sigma,f}$ and $c_3$ we have $\norm{x_n-q}\to 0$ with rate
\begin{equation*}
\Phi^1_{\psi,c_3,\alpha,r,\sigma,f}(\varepsilon):=r\left(N^1_{c_3,\sigma,f}\left(\frac{1}{2}\min\left\{\psi\left(\frac{\varepsilon}{2}\right),\frac{\varepsilon}{\alpha}\right\}\right),2\int_{\varepsilon/2}^{c_3} \frac{dt}{\psi(t)} \right)+1
\end{equation*}
%
We finish the proof by showing that $\norm{z_n-x_n}\to 0$. To this end, observe that
\begin{equation}
\label{eqn:perturbed2}
\begin{aligned}
\norm{z_{n+1}-x_{n+1}}=&\norm{Q_n((1-\alpha_n)z_n+\alpha_nA_nz_n)-Q((1-\alpha_n)x_n+\alpha_nA_nx_n)}\\
\leq &\norm{Q((1-\alpha_n)z_n+\alpha_nA_nz_n) -Q((1-\alpha_n)x_n+\alpha_nA_nx_n)}\\
& +\norm{Q_n((1-\alpha_n)z_n+\alpha_nA_nz_n)-Q((1-\alpha_n)z_n+\alpha_nA_nz_n)}
\end{aligned}
\end{equation}
We consider at the two summands on the right hand side in turn. Fix some $\delta>0$. First of all, since $Q$ is nonexpansive we have
\begin{equation}
\label{eqn:perturbed3}
\begin{aligned}
&\norm{Q((1-\alpha_n)z_n+\alpha_nA_nz_n) -Q((1-\alpha_n)x_n+\alpha_nA_nx_n)}\\
&\leq (1-\alpha_n)\norm{z_n-x_n}+\alpha_n\norm{A_nz_n-A_nx_n}\\
&\leq \norm{z_n-x_n}-\alpha_n\psi(\norm{z_n-x_n})+\alpha_n\delta/2
\end{aligned}
\end{equation}
for all $n\geq \sigma(\delta/2,c_4)$. We now focus on the second term and use the convergence property of $\seq{E_n}$. More specifically, our aim to is show that when $n$ is sufficiently large we have:
\begin{equation}
\label{eqn:perturbed4}
\norm{Q_n((1-\alpha_n)z_n+\alpha_nA_nz_n)-Q((1-\alpha_n)z_n+\alpha_nA_nz_n)}\leq\alpha_n\delta/2
\end{equation}
Noting that
\begin{equation}
\label{eqn:perturbed5}
\begin{aligned}
\norm{(1-\alpha_n)z_n+\alpha_nA_nz_n}&\leq \norm{z_n}+\alpha_n\norm{A_nz_n}\leq c_1+\alpha c_2
\end{aligned}
\end{equation}
we can apply Lemma \ref{lem:sunnyhaus} to the sunny nonexpansive retractions $Q_n$, $Q$ with 
\begin{equation*}
\begin{aligned}
x&:=(1-\alpha_n)z_n+\alpha_nA_nz_n\\
b&:=c_1+\alpha c_2\\
\varepsilon&:=\left(\frac{\alpha_n\delta}{2}\right)^2
\end{aligned}
\end{equation*}
To be more precise, by (\ref{eqn:perturbed5}) we have $\norm{x}\leq c_1+\alpha c_2=b$, by assumption we have $Q0,Q_n0\leq d$, our definition of $R$ satisfies $R=2(2b+d)+1$, and finally the assumption (\ref{eqn:aconv}) together with $a_n\in (0,1)$ implies that $a_n\leq a:=\min\{1,\omega(R,\frac{\varepsilon}{R})\}$ for epsilon as defined above and all $n\geq h(\delta)$. Therefore from $H^\ast[E_n,E,a_n]$ we have $H^\ast[E_n,E,a]$ and thus by Lemma \ref{lem:sunnyhaus} $\norm{Q_nx-Qx}^2\leq \varepsilon$, and instantiating our values for $x$ and $\varepsilon$ we have established (\ref{eqn:perturbed4}) for all $n\geq h(\delta)$.
Now, putting together (\ref{eqn:perturbed2})--(\ref{eqn:perturbed4}) we have
\begin{equation*}
\norm{z_{n+1}-x_{n+1}}\leq \norm{z_n-x_n}-\alpha_n\psi(\norm{z_n-x_n})+\alpha_n\delta
\end{equation*}
whenever $n\geq \sigma(\frac{\delta}{2},c_4)$ and $n\geq h(\delta)$, or alternatively $n\geq N^2_{c_4,\sigma,h}(\delta)$ for
\begin{equation*}
N^2_{c_4,\sigma,h}(\delta):=\max\left\{\sigma\left(\frac{\delta}{2},c_4\right),h(\delta)\right\}
\end{equation*}
Therefore Lemma \ref{res:recineq:first} can be applied again, this time to $\mu_n:=\norm{z_n-x_n}$ and parameters $\psi$, $\alpha$, $r$, $N^2_{c_4,\sigma,h}$ and $c_4$ to establish that $\norm{z_n-x_n}\to 0$ with rate
\begin{equation*}
\Phi^2_{\psi,c_4,\alpha,r,N_1}(\varepsilon):=r\left(N^2_{c_4,\sigma,h}\left(\frac{1}{2}\min\left\{\psi\left(\frac{\varepsilon}{2}\right),\frac{\varepsilon}{\alpha}\right\}\right),2\int_{\varepsilon/2}^{c_4} \frac{dt}{\psi(t)} \right)+1
\end{equation*}
To complete the proof and obtain the final rate of convergence, we assume w.l.o.g. that $\sigma$ is monotone in its second argument and $r$ is monotone in both arguments (in practise this would always be the case, and this harmless assumption allows for a slightly more compact rate of convergence, which would otherwise require additional uses of the maximum operator).
First, using the monotonicity property of $\sigma$ we see that $N_{c_3,c_4,\sigma,f,h}$ as defined in the statement of the theorem satisfies $N_{c_3,c_4,\sigma,f,h}(\delta)\geq N^1_{c_4,\sigma,f}(\delta),N^2_{c_4,\sigma,h}(\delta)$ for all $\delta>0$, and using this together with monotonicity of $r$ for $\Phi(\varepsilon)$ as defined in the statement of the theorem we have (suppressing subscripts) $\Phi(\varepsilon)\geq \Phi^1(\varepsilon/2),\Phi^2(\varepsilon/2)$, and therefore for all $n\geq \Phi(\varepsilon)$ it follows that 
\begin{equation*}
\norm{z_n-q}\leq \norm{x_n-q}+\norm{z_n-x_n}\leq \frac{\varepsilon}{2}+\frac{\varepsilon}{2}=\varepsilon.
\end{equation*} 
Thus $\norm{z_n-q}\to 0$ with the stated rate of convergence.
\end{proof}
\begin{remark}
In special cases, the condition (\ref{eqn:aconv}) can be simplified. For example, if $a_n:=\frac{1}{n}$ and $\alpha_n:=\alpha$ for all $n\in\NN$, then (\ref{eqn:aconv}) automatically holds for 
\begin{equation*}
h(\delta):=\frac{1}{\omega\left(R,(\alpha\delta)^2/R\right)}
\end{equation*}
For a more detailed analysis of what this condition means in uniformly smooth spaces, see Remark \ref{rem:sqrt}.
\end{remark}
%
\subsection{Perturbed schemes for computing fixpoints of $\psi$-weakly contractive mappings}
\label{sec:cs3:concrete}
Theorem \ref{res:perturbed} is very general, and in the case where we replace $\seq{A_n}$ with a single $\psi$-weakly contractive mapping $T$ we obtain a much simpler result as a direct corollary.
\begin{cor}[cf. Theorem 3.5 of \cite{alber-reich-yao:03:weaklycontractive}]
\label{res:perturbed:concrete}
Suppose that $X$ is a uniformly smooth space with a modulus $\tau$, and suppose $\psi$ is a nondecreasing function with $\psi(0)=0$. Let $T:E\to X$ be a $\psi$-weakly contractive mapping and suppose that $\seq{z_n}$ satisfies 
\begin{equation*}
z_{n+1}=Q_n((1-\alpha_n)z_n+\alpha_nTz_n)
\end{equation*}
where $\seq{\alpha_n}$ is some sequence in $(0,\alpha]$ such that $\sum_{n=0}^\infty\alpha_n=\infty$ with rate $r$, and $Q_n:X\to E_n\subseteq E$ is some sequence of sunny nonexpansive retractions. Let $Q:X\to E$ be a sunny nonexpansive retraction and suppose that $H^\ast[E_n,E,a_n]$ holds for all $n\in\NN$, where $\seq{a_n}$ is a sequence in $(0,1)$. Suppose that $d>0$ is such that $Q0\leq d$ and $Q_n0\leq d$ for all $n\in\NN$. Let $q$ be a fixpoint of $T$. Suppose that $\seq{z_n}$ is bounded. Then $\seq{\norm{Tz_n}}$ is also bounded, and we choose $c_1,c_2,c_3>0$ such that $\norm{z_n}\leq c_1$, $\norm{Tz_n}\leq c_2$ and $\norm{z_0-q}\leq c_3$ for all $n\in\NN$. Let $R:=2(2(c_1+\alpha c_2)+d)+1$. Then whenever $h:(0,\infty)\to \NN$ is such that for any $\delta>0$ we have
\begin{equation}
\label{eqn:aconv:concrete}
a_n\leq \omega_\tau\left(R,\frac{(\alpha_n\delta)^2}{4R}\right)
\end{equation}
for all $n\geq h(\delta)$ and $\omega_\tau$ as defined in Remark \ref{rem:unifsmooth}, then $\norm{z_n-q}\to 0$ with rate of convergence
\begin{equation*}
\begin{aligned}
\Phi_{\psi,c_3,\alpha,r,h}(\varepsilon):=r\left(h\left(\frac{1}{2}\min\left\{\psi\left(\frac{\varepsilon}{4}\right),\frac{\varepsilon}{2\alpha}\right\}\right),2\int_{\varepsilon/4}^{2c_3} \frac{dt}{\psi(t)} \right)+1
\end{aligned}
\end{equation*}
\end{cor}

\begin{proof}
We apply Theorem \ref{res:perturbed} with $A_n=T$. Clearly $\sigma(\delta,b)=0$ and also since $\norm{Tq-q}=0$ we have $f(\delta)=0$. Defining $\seq{x_n}$ by $x_0:=z_0$ and
\begin{equation*}
x_{n+1}=Q((1-\alpha_n)x_n+\alpha_nTx_n)
\end{equation*}
we have (using that $Q$ is a nonexpansive retraction):
\begin{equation*}
\begin{aligned}
\norm{x_{n+1}-q}&\leq (1-\alpha_n)\norm{x_n-q}+\alpha_n\norm{Tx_n-Tq}\\
&\leq \norm{x_n-q}-\alpha_n\psi(\norm{x_n-q})\\
&\leq \norm{x_n-q}
\end{aligned}
\end{equation*}
and therefore $\norm{x_n-q}\leq \norm{x_0-q}=\norm{z_0-q}\leq c_3$ for all $n\in\NN$. Thus boundedness of $\seq{x_n}$ can be derived in this case and $c_3$ is an upper bound for $\seq{\norm{x_n-q}}$, which also implies that $\norm{x_n-z_n}\leq c_4$ for $c_4:=2c_3$. Then the result follows from putting this data into Theorem \ref{res:perturbed}, and observing that in this case, $N(\delta)=h(\delta)$ and $\max\{c_3,c_4\}=2c_3$.
\end{proof}

\begin{remark}
\label{rem:sqrt}
Corollary \ref{res:perturbed:concrete} forms a quantitative version of Theorem 3.5 of \cite{alber-reich-yao:03:weaklycontractive}. The only assumption which does not obviously translate is the condition (\ref{eqn:aconv:concrete}). Let $g_X(\delta):=\rho_X(\delta)/\delta$ where $\rho_X$ is the modulus of smoothness of $X$ (cf. Section \ref{sec:basic:smooth}). It is known that when $X$ is a uniformly smooth then for $\norm{x},\norm{y}\leq d$ then
\begin{equation*}
\norm{Jx-Jy}\leq 8dg_X(16Ld^{-1}\norm{x-y})
\end{equation*}
where $1<L<1.7$ is the so-called Figiel constant \cite{figiel:76:constant}. In this sense, modulo some constants, we can informally view $g_X^{-1}$ as a modulus of continuity for the duality mapping, and thus (\ref{eqn:aconv:concrete}) would correspond to 
\begin{equation*}
g_X(a_n)\leq (\alpha_n\delta)^2
\end{equation*}
and thus $h$ would be a modulus of convergence for
\begin{equation*}
\frac{\sqrt{g_X(\alpha_n)}}{\alpha_n}\to 0 \mbox{ \ \ as \ \ }n\to\infty
\end{equation*}
which is the limiting condition given in Theorem 3.5 of \cite{alber-reich-yao:03:weaklycontractive}.
\end{remark}

\section{Concluding remarks}
\label{sec:conc}

There are several ways in which this work could be extended. The most obvious direction for future research is to consider further variants and generalisations of asymptotically weakly contractive mappings (or indeed closely related families of mappings such as the asymptotically contractive mappings, which have already been studied in applied proof theory in \cite{kohlenbach-lambov:04:asymptotically:nonexpansive,kohlenbach-leustean:10:asymp:nonexpansive:hyperbolic} but for which several interesting convergence results not hitherto analysed from a proof theoretic perspective are established in the second part of \cite{alber-chidume-zegeye:06:nonexpansive}). In this paper we have selected a handful of case studies to exemplify our approach, but a great deal of work on mappings of weakly contractive and related type has been done in the last decade, much of which may be amenable to the kind of quantitative analysis we carry out here.

Secondly, throughout this paper we have assumed that $q$ acts as a fixpoint of some limit of the sequence $\seq{A_n}$ (this is typically implicit in our main theorems but explicit in their corollaries). However, it is natural to ask whether this assumption can be weakened in some way and replaced by the existence of arbitrary $\varepsilon$-approximate fixpoints. This is done in the related papers \cite{kohlenbach-koernlein:11:pseudocontractive,kohlenbach-lambov:04:asymptotically:nonexpansive} and there are many other examples in applied proof theory, and here could potentially then lead to moduli of uniqueness of fixpoints of asymptotically weakly contractive mappings.

Finally, along with \cite{kohlenbach-koernlein:11:pseudocontractive,kohlenbach-lambov:04:asymptotically:nonexpansive,kohlenbach-leustean:10:asymp:nonexpansive:hyperbolic,kohlenbach-powell:20:accretive} our paper represents a further instance of common situation in applied proof theory where concrete numerical results have been obtained through the quantitative analysis of abstract recursive inequalities (which here form the main topic of Section \ref{sec:recursive}). We believe it would be useful to undertake a comprehensive quantitative study of these abstract recursive schemes, bringing together known results and establishing new ones. Not only would this be of interest in its own right, but it would provide a valuable repository of quantitative lemmas which could then be applied in concrete situations.\medskip

\noindent\textbf{Acknowledgements.} The authors thank Ulrich Kohlenbach for reading an earlier draft of this paper and making a number of insightful comments. The second author is a Marie Sk\l odowska-Curie fellow of INdAM.

\bibliographystyle{acm}
\bibliography{/Users/trjp20/Dropbox/tpbiblio}

\begin{thebibliography}{10}

\bibitem{alber:83:inequalities}
{\sc Alber, Y.}
\newblock Recurrence relations and variational inequalities.
\newblock {\em Soviet Mathematics Doklady 27\/} (1983), 511--517.
\newblock (Russian).

\bibitem{alber-chidume-zegeye:06:nonexpansive}
{\sc Alber, Y., Chidume, C., and Zegeye, H.}
\newblock Approximating fixed points of total asymptotically nonexpansive
  mappings.
\newblock {\em Fixed Point Theory and Applications 2006\/} (2006), 1687--1812.

\bibitem{alber-guerredelabriere:97:weaklycontractive}
{\sc Alber, Y., and Guerre-Delabriere, S.}
\newblock Principle of weakly contractive maps in {H}ilbert spaces.
\newblock In {\em New Results in Operator Theory and its Applications\/}
  (1997), I.~Gohberg and Y.~Lyubich, Eds., vol.~98, pp.~7--22.

\bibitem{alber-guerredelabriere:01:projection}
{\sc Alber, Y., and Guerre-Delabriere, S.}
\newblock On the projection methods for fixed point problems.
\newblock {\em Analysis: International mathematical journal of analysis and its
  applications 21\/} (2001), 17--39.

\bibitem{alber-iusem:01:subgradient}
{\sc Alber, Y., and Iusem, A.~N.}
\newblock Extension of subgradient techniques for nonsmooth optimization in
  {B}anach spaces.
\newblock {\em Set-Valued Analysis 9\/} (2001), 315--335.

\bibitem{alber-reich:94:panamerican}
{\sc Alber, Y., and Reich, S.}
\newblock An iterative method for solving a class of nonlinear operator
  equations in {B}anach spaces.
\newblock {\em Panamerican Mathematical Journal 4\/} (1994), 39--54.

\bibitem{alber-reich-yao:03:weaklycontractive}
{\sc Alber, Y., Reich, S., and Yao, J.-C.}
\newblock Iterative methods for solving fixed-point problems with
  nonself-mappings in {B}anach spaces.
\newblock {\em Abstract and Applied Analysis 2003}, 4 (2003).

\bibitem{bridges-etal:92:contractive}
{\sc Bridges, D., Richman, F., Julian, W.~H., and Mines, R.}
\newblock Extensions and fixed points of contractive maps in $\mathbb{R}^n$.
\newblock {\em Journal of Mathematical Analysis and Applications 165}, 2
  (1992), 438--456.

\bibitem{briseid:07:asymptotic}
{\sc Briseid, E.}
\newblock A rate of convergence for asymptotic contractions.
\newblock {\em Journal of Mathematical Analysis and Applications 330\/} (2007),
  364--376.

\bibitem{bruck:73:nonexpansive:projections}
{\sc Bruck, R.~E.}
\newblock Nonexpansive projections on subsets of {B}anach spaces.
\newblock {\em Pacific Journal of Mathematics 47\/} (1976), 341--355.

\bibitem{chidume-zegeye-aneke:02:dweakly}
{\sc Chidume, C., Zegeye, H., and Aneke, S.~J.}
\newblock Approximation of fixed points of weakly contractive nonself maps in
  {B}anach spaces.
\newblock {\em Journal of Mathematical Analysis and Applications 270\/} (2002),
  189--199.

\bibitem{edelstein:62:contractive}
{\sc Edelstein, M.}
\newblock On fixed and periodic points under contractive mappings.
\newblock {\em Journal of the London Mathematical Society 37\/} (1962), 73--79.

\bibitem{figiel:76:constant}
{\sc Figiel, T.}
\newblock On the moduli of convexity and smoothness.
\newblock {\em Studia Mathematica 56}, 2 (1976), 121--155.

\bibitem{gerhardy:06:kirk}
{\sc Gerhardy, P.}
\newblock A quantitative version of {K}irk’s fixed point theorem for
  asymptotic contractions.
\newblock {\em Journal of Mathematical Analysis and Applications 316\/} (2006),
  339--345.

\bibitem{goebel-reich:84:book}
{\sc Goebel, K., and Reich, S.}
\newblock {\em Uniform convexity, hyperbolic geometry, and nonexpansive
  mappings}, vol.~83 of {\em Monographs and Textbooks in Pure and Applied
  Mathematics}.
\newblock Marcel Dekker, New York, 1984.

\bibitem{kirk:03:asymptotic-contractions}
{\sc Kirk, W.~A.}
\newblock Fixed points of asymptotic contractions.
\newblock {\em Journal of Mathematical Analysis and Applications 277\/} (2003),
  645--650.

\bibitem{kohlenbach:08:book}
{\sc Kohlenbach, U.}
\newblock {\em {Applied Proof Theory: Proof Interpretations and their Use in
  Mathematics}}.
\newblock Springer Monographs in Mathematics. Springer, 2008.

\bibitem{kohlenbach:19:nonlinear:icm}
{\sc Kohlenbach, U.}
\newblock Proof-theoretic methods in nonlinear analysis.
\newblock In {\em Proceedings of the International Congress of Mathematicians
  2018}, vol.~2. World Scientific, 2019, pp.~61--82.

\bibitem{kohlenbach-koernlein:11:pseudocontractive}
{\sc Kohlenbach, U., and K\"ornlein, D.}
\newblock Effective rates of convergence for {L}ipschitzian pseudocontractive
  mappings in general {B}anach spaces.
\newblock {\em Nonlinear Analysis 74\/} (2011), 5253--5267.

\bibitem{kohlenbach-lambov:04:asymptotically:nonexpansive}
{\sc Kohlenbach, U., and Lambov, B.}
\newblock Bounds on iterations of asymptotically quasi-nonexpansive mappings.
\newblock In {\em Proceedings of the International Conference on Fixed Point
  Theory and Applications\/} (2004), Yokohama Publishers, pp.~143--172.

\bibitem{kohlenbach-leustean:10:asymp:nonexpansive:hyperbolic}
{\sc Kohlenbach, U., and Leu\c{s}tean, L.}
\newblock Asymptotically nonexpansive mappings in uniformly convex hyperbolic
  spaces.
\newblock {\em Journal of the European Mathematical Society 12}, 1, 71--92.

\bibitem{kohlenbach-leustean:12:modulus}
{\sc Kohlenbach, U., and Leu\c{s}tean, L.}
\newblock On the computation content of convergence proofs via {B}anach limits.
\newblock {\em Philosophical Transactions of the Royal Society A 370\/} (2012),
  3449--3463.

\bibitem{kohlenbach-lopezacedo-nicolae:21:lionman}
{\sc Kohlenbach, U., L\'opez-Acedo, G., and Nicolae, A.}
\newblock A uniform betweenness property in metric spaces and its role in the
  quantitative analysis of the ``{L}ion-{M}an'' game.
\newblock {\em Pacific Journal of Mathematics 310}, 1 (2021), 181--212.

\bibitem{kohlenbach-oliva:03:systematic}
{\sc Kohlenbach, U., and Oliva, P.}
\newblock Proof mining: a systematic way of analyzing proofs in mathematics.
\newblock {\em Proceedings of the Steklov Institute of Mathematics 242\/}
  (2003), 136--164.

\bibitem{kohlenbach-powell:20:accretive}
{\sc Kohlenbach, U., and Powell, T.}
\newblock Rates of convergence for iterative solutions of equations involving
  set-valued accretive operators.
\newblock {\em Computers and Mathematics with Applications 80\/} (2020),
  490--503.

\bibitem{kohlenbach-sipos:21:sunny}
{\sc Kohlenbach, U., and Sipo\c{s}, A.}
\newblock The finitary content of sunny nonexpansive retractions.
\newblock {\em Communications in Contemporary Mathematics 23}, 01 (2021),
  1950093.

\bibitem{koernlein:15:halpern}
{\sc K\"ornlein, D.}
\newblock Quantitative results for {H}alpern iterations of nonexpansive
  mappings.
\newblock {\em Journal of Mathematical Analysis and Applications 428\/} (2015),
  1161--1172.

\bibitem{powell:20:tauberian}
{\sc Powell, T.}
\newblock A note on the finitization of {A}belian and {T}auberian theorems.
\newblock {\em Mathematical Logic Quarterly 66}, 3 (2020), 300--310.

\bibitem{simmons:towsner:19:polyrings}
{\sc Simmons, W., and Towsner, H.}
\newblock Proof mining and effective bounds in differential polynomial rings.
\newblock {\em Advances in Mathematics 343\/} (2019), 567--623.

\end{thebibliography}

\end{document}